\documentclass[11pt]{amsart}
  \usepackage[margin=3.3cm]{geometry}
\usepackage{amsfonts,amsmath,amssymb,color,esint}
\usepackage{enumerate}
\usepackage{graphicx}
\usepackage{tikz}
\usetikzlibrary{decorations.markings}
\usetikzlibrary{plotmarks}
\usetikzlibrary{patterns}
\numberwithin{equation}{section}

\usepackage{amsthm}

\usepackage[babel=true,kerning=true]{microtype}

\usepackage{amsthm,leqno}
\usepackage{hyperref}

\newtheorem{theo}{Theorem}[section]

\newtheorem{rem}{Remark}[section]

\newtheorem{cl}{Claim}[section]
\newtheorem{prop}{Proposition}[section]

\newcommand{\eps}{\varepsilon}

\newcommand{\R}{\mathbb{R}}

\begin{document}

\title[Regularity estimates on harmonic eigenmaps]{Regularity estimates on harmonic eigenmaps with arbitrary number of coordinates}
\author{Romain Petrides} 
\address{Romain Petrides, Universit\'e Paris Cité, Institut de Math\'ematiques de Jussieu - Paris Rive Gauche, b\^atiment Sophie Germain, 75205 PARIS Cedex 13, France}
\email{romain.petrides@imj-prg.fr}

\begin{abstract}
We revisit the well-established regularity estimates on harmonic maps on surfaces to question their independence with respect to the dimension of the target manifold. We are mainly interested in harmonic maps into target ellipsoids, that we call \textit{Laplace harmonic eigenmaps}. These maps are related to critical metrics in the context of eigenvalue optimization. The tools that we gather here are useful to handle convergence of almost critical metrics via Palais-Smale sequences of (almost harmonic) eigenmaps. They could also be a preliminary step for a general regularity theory for critical points of infinite combinations of eigenvalues.
\end{abstract}

\maketitle

On a Riemannian manifold $(\Sigma,g)$, a harmonic map $\Phi : \Sigma \to M$ is defined as a critical point of the Dirichlet energy 
$$ E(\Phi) = \frac{1}{2} \int_\Sigma \vert \nabla \Phi\vert_g^2 dv_g $$
under the constraint that the values of the map $\Phi$ belong to a manifold $M$ embedded in $\mathbb{R}^m$. The Euler-Lagrange equation is
$$ \Delta_g \Phi \in \left(T_{\Phi(x)}M\right)^{\perp}.  $$
The study of regularity estimates for harmonic maps has been intensive for years. We assume that $\Sigma$ is a surface (a manifold of dimension 2). Regularity results and fine regularity estimates are now well established in this case thanks to \cite{helein,helein2,riviere}. Among these classical references, regularity estimates on harmonic maps are almost never written with constants that are independent of the dimension of $M$ or of $m$. The goal of the current paper is to investigate on this issue.

Our motivation is the characterization in \cite{petrides-4,pt} of critical metrics for combinations of eigenvalues in a conformal class with harmonic maps $\Phi : \Sigma \to \mathcal{E}_\Lambda$ into ellipsoids
$$ \mathcal{E}_{\Lambda} = \{ \lambda_1 x_1^2 + \cdots + \lambda_n x_n^2 = 1 \}$$
where $\Lambda = (\lambda_1,\cdots,\lambda_n)$ that satisfy the equation 
$$ \Delta_g \Phi = \frac{\vert \nabla \Phi \vert^2_{\Lambda,g}}{\vert \Lambda \Phi\vert^2} \Lambda \Phi$$
where if $\Phi = (\phi_1,\cdots,\phi_n)$,
$$ \vert \Lambda \Phi\vert^2 = \sum_{i=1}^n  \lambda_i^2 \phi_i^2 \text{ and } \vert \nabla \Phi \vert^2_{\Lambda,g} =  \sum_{i=1}^n \lambda_i \vert \nabla \phi_i \vert^2_{g} $$
and the coordinate functions $\phi_k$ are eigenfunctions with respect to the metric (which is smooth with conical singularities) 
$$\frac{\vert \nabla \Phi \vert^2_{\Lambda,g}}{\vert \Lambda \Phi\vert^2}g$$
associated to the eigenvalue $\lambda_k$. Such a map $\Phi$ is called a \textit{Laplace harmonic eigenmap}. Notice that defining $\Phi$ is equivalent to defining a family of eigenfunctions such that a positive linear combination of their squares is a constant function.

More precisely, for a fixed metric $g_0$ on $\Sigma$ and its conformal class $[g_0]$, we let
$$ E : g \in [g_0] \mapsto F(\bar{\lambda}_1(g),\cdots,\bar{\lambda}_m(g)) $$
be a finite combination of eigenvalue functionals where $F : \mathbb{R}_+^m \to \mathbb{R}$ is a $\mathcal{C}^1$ map and for $k \in \mathbb{N}^\star$, $ \bar{\lambda}_k(g) $ is the $k$-th non-zero Laplace eigenvalue of $(\Sigma,g)$ multiplied by the area of $(\Sigma,g)$. In \cite{pt}, $g$ is defined as a critical metric of $E$ if $0$ belongs to the Clarke sub-differential or the Clarke super-differential of $E$. Then 

\begin{theo}[\cite{petrides-4,pt}] Let $g \in [g_0]$. We assume that the sign of $\partial_i F(\bar{\lambda}_1(g),\cdots,\bar{\lambda}_m(g))$ is independent of $1\leq i \leq m$. If $g$ is critical for $E$, then there is a harmonic map $\Phi : \Sigma \to \mathcal{E}_\Lambda \subset \R^n$ where $\Lambda = (\bar{\lambda}_1(g),\cdots,\bar{\lambda}_{m-1}(g),\bar{\lambda}_m(g), \cdots, \bar{\lambda}_m(g)) \in\R^n $.
Moreover, $\left( \phi_k \right)_{1\leq k \leq n}$ is a $L^2(g)$ orthogonal family of eigenfunctions associated to $\bar{\lambda}_k(g)$ and we have 
$$ g = \frac{\vert \nabla \Phi \vert^2_{\Lambda,g_0}}{\vert \Lambda \Phi\vert^2} g_0. $$
\end{theo}
Notice that in \cite{pt}, we prove more properties on the coordinates $\phi_k$ and as a reciprocal, all the harmonic maps into some ellipsoid can be associated to a critical metric of some combination of Laplace eigenvalues $E$.

The questions of regularity estimates on critical metrics $g$ and compactness of the set of critical metrics are related to local regularity estimates on the conformal factor
$$\beta = \frac{\vert \nabla \Phi \vert^2_{\Lambda,g_0}}{\vert \Lambda \Phi\vert^2}$$ 
with constants that are independent of the number of coordinates $n$ of $\Phi$ and - as much as we can - independent of the parameters $\Lambda$ of the ellipsoid. Since we look for local estimates, by conformal invariance of the harmonic map equation and by the use of the isothermal coordinates in dimension 2, we can assume that $\Sigma = \mathbb{D}$ and that $g_0$ is the Euclidean metric.

In the previous theorem, $n$ is not \textit{a priori} bounded by the choice of the functional $E$ (or the maximal index for eigenvalues $m$ that appears in $F$). Indeed, $n$ is bounded by $m-1 + \mu_m(g)$ where $\mu_m(g)$ is the multiplicity of $m$-th non-zero eigenvalue of $g$. We \textit{a posteriori} know bounds of $\mu_m(g)$ that only depend on $m$ and the topology of $\Sigma$ (see \cite{cheng}). However, this \textit{a posteriori} result is global and needs a control on the topology of $\Sigma$ : in the case $\mathcal{E}_\Lambda$ is a sphere, the author in \cite{song} deals with sequences of harmonic maps $\Phi_j : \Sigma_j \to \mathbb{S}^{n_j}$ with variable topology of the surface as $j \to +\infty$.

Furthermore, in \cite{KS,petrides6,kkms}, the authors wrote various methods to prove existence of critical metrics by optimization of functionals $E$. A similarity between all these methods is the necessity to prove that sequences of "conformal factors" associated to \textit{almost} harmonic maps with arbitrary number of coordinates are compact. They need local regularity estimates to prove their compactness. More precisely, even with $\Sigma$ fixed, they cannot use bounds on the multiplicity of eigenvalues for different technical reasons. In the constructions of \cite{KS,kkms}, the authors do not \text{a priori} know bounds on the indices of eigenvalues corresponding to the coordinates of their harmonic maps. In \cite{petrides6}, the coordinates of \textit{almost} harmonic maps that appear in the Palais-Smale sequences are only $H^1$ functions whereas proofs of multiplicity bounds need eigenfunctions to be at least continuous.

Finally, it would also be interesting to understand the behavior of a sequence
\begin{equation} \label{conffactorseq} \beta_j = \frac{\vert \nabla \Phi_j \vert^2_{\Lambda_j ,g_0}}{\vert \Lambda_j \Phi_j\vert^2} \end{equation}
where $\Phi_j : (\Sigma,g_0) \to \mathcal{E}_{\Lambda_j}$ is a harmonic map and $\Lambda_j \in \mathbb{R}^{n_j}$: there are non-bounded indices of eigenvalues that appear in $\Lambda_j$ as $j\to +\infty$. Indeed, the compactness of such a sequence would imply the existence of minimizers for infinite combinations of eigenvalues, e.g the zeta-function on eigenvalues for $s >1$
$$ \zeta(s) : g \in [g_0] \mapsto \sum_{k=1}^{+\infty} \bar{\lambda}_k^{-s}(g) $$
where $\Phi_j$ is a harmonic map associated to a critical metric of $\zeta_j(s) : g\mapsto\sum_{k=1}^{j} \bar{\lambda}_k^{-s}(g)$.

\medskip

All the previous questions are also addressed for free boundary harmonic maps into ellipsoids of $\mathbb{R}^m$. If $(\Sigma,g)$ is a compact surface with boundary, a free boundary harmonic map $\Phi : (\Sigma,\partial \Sigma) \to (\R^m,M)$ is defined as a critical point of the Dirichlet energy 
$$ E(\Phi) = \frac{1}{2} \int_\Sigma \vert \nabla \Phi\vert_g^2 dv_g $$
under the constraint that the values of the map $\Phi$ on $\partial \Sigma$ belong to a manifold $M$ embedded in $\mathbb{R}^m$. The Euler-Lagrange equation is
$$ \begin{cases} \Delta_g \Phi = 0 \text{ in } \Sigma \\
\partial_\nu \Phi \in \left(T_{\Phi(x)}M\right)^{\perp} \text{ on } \partial \Sigma
\end{cases}. $$
In this case, we ask for the regularity estimates proved in \cite{scheven,LP,JLZ} with constants that are independent of the dimension of $M$ or of $m$. As previously, there is a characterization of free boundary harmonic maps into ellipsoids (called \textit{Steklov harmonic eigenmaps}) via critical metrics of eigenvalues. For $\sigma = (\sigma_1,\cdots,\sigma_n)$, a map $\Phi : (\Sigma,\partial\Sigma) \to \left(co(\mathcal{E}_\sigma), \mathcal{E}_\sigma\right)$ is a Steklov harmonic eigenmap if
$$ \begin{cases} \Delta_g \Phi = 0 \text{ in } \Sigma \\
\partial_\nu \Phi = \left(\partial_\nu \Phi \cdot \Phi \right) \sigma \Phi \text{ on } \partial \Sigma
\end{cases}. $$
where the coordinate functions $\phi_k$ are Steklov eigenfunctions with respect to the smooth metric
$$ e^{2u} g \text{ where } e^u = \partial_{\nu_{g_0}} \Phi \cdot \Phi \text{ on } \partial \Sigma $$
associated to the eigenvalue $\sigma_k$ where $u$ is any smooth extension in $\Sigma$ of the positive function $\ln\left( \partial_{\nu_{g_0}} \Phi \cdot \Phi\right) $.

More precisely, for a fixed metric $g_0$ on $\Sigma$ and its conformal class $[g_0]$, we let
$$ E : g \in [g_0] \mapsto F(\bar{\sigma}_1(g),\cdots,\bar{\sigma}_m(g)) $$
be a finite combination of eigenvalue functionals where $F : \mathbb{R}_+^m \to \mathbb{R}$ is a $\mathcal{C}^1$ map and for $k \in \mathbb{N}^\star$, $ \bar{\sigma}_k(g) $ is the $k$-th non-zero Steklov eigenvalue of $(\Sigma,g)$ multiplied by the length of $(\partial \Sigma,g)$. Then 

\begin{theo}[\cite{petrides-5,pt}] Let $g \in [g_0]$. We assume that the sign of $\partial_i F(\bar{\sigma}_1(g),\cdots,\bar{\sigma}_m(g))$ is independent of $1\leq i \leq m$. If $g$ is critical for $E$, then there is a free-boundary harmonic map $\Phi : (\Sigma,\partial \Sigma) \to \left(co(\mathcal{E}_\sigma),\mathcal{E}_\sigma\right) \subset \R^n$ where $\sigma = (\bar{\sigma}_1(g),\cdots,\bar{\sigma}_{m-1}(g),\bar{\sigma}_m(g), \cdots, \bar{\sigma}_m(g)) \in\R^n $.
Moreover, $\left( \phi_k \right)_{1\leq k \leq n}$ is a $L^2(\partial\Sigma,g)$ orthogonal family of eigenfunctions associated to $\bar{\sigma}_k(g)$ and we have 
$$ g = e^{2u} g_0 \text{ where } e^u = \partial_{\nu_{g_0}} \Phi \cdot \Phi \text{ on } \partial \Sigma $$
\end{theo}
In this context, we look for local regularity estimates on the function defined on $\partial \Sigma$: $\partial_{\nu_{g_0}} \Phi \cdot \Phi$ or equivalently $ \frac{\vert \partial_{\nu_{g_0}} \Phi \vert}{\vert\sigma \Phi\vert}$. In this case, choosing isothermal coordinates, we can assume that $\Sigma = \mathbb{D}^+$, that $[-1,1]\times \{0\}$ is the free boundary portion of $\partial \mathbb{D}^+$ and that $g_0$ is the Euclidean metric.

\medskip

The paper is divided into two parts. In the first part, we use arguments based on the second order equation on $ \vert \nabla \Phi \vert^2$ deduced from Bochner's formula to obtain the standard $\eps$-regularity result for harmonic maps $\Phi : \mathbb{D}\to \mathcal{E}_\Lambda$
\begin{equation} \label{epreg} \int_\mathbb{D} \vert \nabla \Phi \vert^2 \leq \eps_\alpha \Rightarrow \forall x \in \mathbb{D}, \vert \nabla \Phi(x) \vert^2 \leq C_\alpha \frac{\int_\mathbb{D} \vert \nabla \Phi \vert^2}{\left(1-\vert x \vert\right)^2} \end{equation}
or free-boundary harmonic maps $\Phi : (\mathbb{D}_+,[-1,1]\times \{0\}) \to \left(co( \mathcal{E}_\sigma),\mathcal{E}_\sigma\right)$
\begin{equation} \label{epreg2} \int_{\mathbb{D}_+} \vert \nabla \Phi \vert^2 \leq \eps_\alpha \Rightarrow \forall x \in \mathbb{D}_+, \vert \nabla \Phi(x) \vert^2 \leq C_\alpha \frac{\int_{\mathbb{D}_+} \vert \nabla \Phi \vert^2}{\left(1-\vert x \vert\right)^2} \end{equation}
where $\eps_\alpha$ and $C_\alpha$ are constants that depend only on the elongation $\alpha$ of the target ellipsoid (see Theorem \ref{theo:epsregharm1} and Theorem \ref{theomainSteklov}). Even if the way to prove Theorem \ref{theo:epsregharm1} is well-known, was already noticed in \cite{KS} in the case of a target sphere and is a consequence of Theorem \ref{theo:Linftyestimateofgradpsi} below, it's worth writting a short proof. However, Theorem \ref{theomainSteklov} is new even in the case when the target ellipsoid is a sphere. Indeed, in \cite[Lemma 4.10]{kkms}, the authors needed an extra assumption
$$  \int_{[-1,1]\times \{0\}} \vert \Phi \cdot \partial_\nu \Phi \vert \leq \eps_\alpha$$
that we remove here thanks to symmetrization arguments on $\Phi$ inspired by \cite{scheven,LP,JLZ}. Theorem \ref{theo:epsregharm1} and Theorem \ref{theomainSteklov} are used in \cite{petrides6} to prove compactness of Palais-Smale sequences of spectral functionals.

In the second part of the paper, we make the constants independent of the elongation of the ellipsoid. However, we pay a price: the smallness assumption holds on another energy $\int_{\mathbb{D}} \left\vert \nabla \frac{\Lambda \Phi}{\vert \Lambda \Phi \vert}  \right\vert^2$
where $\nu = \frac{\Lambda \Phi}{\vert \Lambda \Phi \vert}$ is nothing but the unit normal to the ellipsoid $\mathcal{E}_\Lambda$ at $\Phi$. 
\begin{theo} \label{theo:Linftyestimateofgradpsi}
There is $\eps_0 >0$ and a constant $C$ such that for any $m \in \mathbb{N}^*$ and any $\Psi : \mathbb{D} \to \mathcal{E}_{\Lambda}$ a Laplace harmonic eigenmap into the ellipsoid
$ \mathcal{E}_{\Lambda} = \{ X \in \R^{m} ; \left\vert X \right\vert_{\Lambda} = 1 \} $
where $\Lambda = (\lambda_1,\cdots,\lambda_{m})$ and $\lambda_1\leq \cdots \leq \lambda_m$, such that setting $\nu = \frac{\Lambda \Phi}{\vert \Lambda \Phi \vert}$,
we assume that
$$ \int_\mathbb{D} \left\vert \nabla \nu \right\vert^2 \leq \eps_0.$$
Then
\begin{equation} \label{eqconsequenceclaimalphaplus} \forall x \in \mathbb{D},  \left\vert \nabla \Phi(x) \right\vert^2  
 \leq C \frac{ \int_{\mathbb{D}} \left\vert \nabla \Phi \right\vert^2}{\left(1-\vert x \vert\right)^2}  \end{equation}
and
\begin{equation} \label{eqconsequenceclaimalpha2plus}  \sum_i \int_{\mathbb{D}}  (\Delta \Phi_i(x))^2 (1-\vert x \vert)^2  \leq  C \int_{\mathbb{D}} \left\vert \nabla \Phi \right\vert^2  \int_\mathbb{D} \left\vert \nabla \nu \right\vert^2.  \end{equation}
\end{theo}
Of course Theorem \ref{theo:Linftyestimateofgradpsi} implies \eqref{epreg}. This assumption is very natural if we write the harmonic map equation as Rivière's equation
\begin{equation} \label{eqrivieresystem} \Delta \Phi = \Omega \cdot \nabla \Phi \end{equation}
where $\Omega = \nu \cdot \nabla \nu^T - \nabla \nu \cdot \nu^T \in L^2\left(\mathbb{D},so(m)\right)$. Thanks to Uhlenbeck's gauge theory, Rivière proved in \cite{riviere} $\eps$-regularity of solutions of \eqref{eqrivieresystem} with smallness of the following energy
$$ \int_\mathbb{D}\vert \Omega \vert^2 = \sum_{1\leq i,j \leq m} \int_\mathbb{D} \Omega_{i,j}^2 = 2 \sum_{i=1}^m \int_\mathbb{D} \vert \nabla \nu_i \vert^2 = 2 \int_{\mathbb{D}} \vert \nabla \nu \vert^2.$$
Then, we carefully adapt this theory thanks to \cite{schikorra,sharptopping,daliopalmurella,khomrutaischikorra} in order to obtain constants that are independent of the dimension of the target manifold. However, in this theory without a pointwise bound $\vert \Omega \vert^2 \leq C \vert \nabla \Psi \vert^2$, the bootstrap stops: we only have $L^p$ estimates for $\vert \nabla \Phi \vert$ for any $p < +\infty$. Fortunately, for harmonic map equations, we have another hidden equation
$$ \partial_{\bar z} \alpha = \omega \alpha $$
where $\vert \alpha \vert^2 = \vert \nabla \Phi \vert^2$ and $\omega \in L^2(\mathbb{D},so(m))$ that satisties $\sum_{i,j} \Vert \omega \Vert_{ L^{2,1}}^2 \leq C \int_{\mathbb{D}} \vert \Omega \vert^2$, so that we can obtain $\vert \nabla \Phi \vert^2 \in L^\infty$ following carefully an idea of \cite[page 182]{helein2}.

Notice that Theorem \ref{theo:Linftyestimateofgradpsi} holds with harmonic maps into arbitrary codimension $1$ submanifolds of $\mathbb{R}^n$ for free. In this special case, we can interpret the $L^2$ bound on $\nabla \nu$ as a sufficient condition to have a $W^{1,\infty} \cap W^{2,2}$ bound on $\Phi$ that is coarsely the minimal assumption to define weak immersions of $L^2$ second fundamental form. 

As an application of Theorem \ref{theo:Linftyestimateofgradpsi}, we obtain local $L^2$ estimates on $\langle \nabla \Phi,\nabla \nu \rangle = \vert \Lambda \Phi \vert \beta$ where $\beta$ is the conformal factor associated to Laplace harmonic eigenmaps \eqref{conffactorseq} :
\begin{equation} \label{eq:epsregconffactor} \int_\mathbb{D} \left\vert \nabla \nu \right\vert^2 \leq \eps_0 \Rightarrow \left\Vert \langle \nabla \Phi,\nabla \nu \rangle \right\Vert_{L^2\left(\mathbb{D}_{1-r}\right)}^2 \leq  \frac{C}{r^2} \eps_0\int_\mathbb{D} \vert \nabla \Phi \vert^2 \end{equation}
With respect to the conformal factor $\beta$, we deduce that
\begin{equation} \label{eq:epsregconffactorp} \int_\mathbb{D}  \left\vert \nabla \nu \right\vert^2 \leq \eps_0 \Rightarrow \left\Vert \beta \right\Vert_{L^2\left(\mathbb{D}_{\frac{1}{2}}\right)}^2 \leq \frac{C_0 \eps_0}{\left(\min_i \lambda_i\right)^2} \int_{\mathbb{D}} \left\vert \nabla \Phi \right\vert^2. \end{equation}
The estimate 
\eqref{eq:epsregconffactorp} is a starting point to understand compactness properties of sets of conformal factors associated to harmonic eigenmaps. We also believe that the present paper could be a starting point of a general regularity theory for  harmonic eigenmaps into Hilbert spaces, that would correspond to critical points of infinite combinations of eigenvalues.

\section{Regularity estimates for non elongated target ellipsoids of arbitrary dimension.}

In this section, we use the already known \text{a priori} regularity of eigenmaps into finite-dimensional ellipsoids to obtain classical $\eps$-regularity results that do not depend on the dimension of the ellipsoid, with the assumption that the ellipsoid is not too elongated. 

\subsection{Laplace harmonic eigenmaps}

\begin{theo} \label{theo:epsregharm1}
For any $\alpha>1$, there is $\eps_\alpha>0$ and $C_\alpha > 0$ such that for every $k\in \mathbb{N}$ and $\Lambda = (\lambda_1,\cdots,\lambda_k)$ with
$$ \max_{1\leq i \leq n} \lambda_i \leq \alpha \text{ and } \min_{1\leq i \leq n} \lambda_i \geq \alpha^{-1}, $$
such that if $\Phi : \mathbb{D} \to \mathcal{E}_{\Lambda}$ is a harmonic map then
$$ \int_{\mathbb{D}} \left\vert \nabla \Phi \right\vert^2 \leq \eps_\alpha \Rightarrow \forall x \in \mathbb{D}, \vert \nabla \Phi(x) \vert^2 \leq C_\alpha \frac{\int_{\mathbb{D}} \vert \nabla \Phi \vert^2}{\left( 1-\vert x \vert\right)^2} $$
\end{theo}

\begin{proof} A solution of the harmonic map equation $\Delta \Phi = \frac{\vert \nabla \Phi \vert_\Lambda^2}{\vert \Lambda \Phi\vert^2} \Lambda \Phi$ satisfies
$$ \frac{1}{2} \Delta \vert \nabla \Phi \vert^2 = - \vert \nabla^2 \Phi \vert^2 + \nabla \Delta\Phi \nabla \Phi = - \vert \nabla^2 \Phi \vert^2 + \nabla\left( \frac{\vert \nabla \Phi \vert_\Lambda^2}{\vert \Lambda \Phi \vert^2} \Lambda \Phi\right) \nabla \Phi = - \vert \nabla^2 \Phi \vert^2 + \frac{\vert \nabla \Phi \vert_\Lambda^4}{\vert \Lambda \Phi\vert^2} $$
where we used $\Lambda \Phi \nabla \Phi = \nabla \frac{\vert\Phi \vert_\Lambda^2}{2} = 0$. We will prove existence of constants $\epsilon_0>0$ and $C>0$ such that for all $p \in \mathbb{D}$ and all $r < 1 - \vert p \vert $,
$$ \forall y \in \mathbb{D}_r(p), F_{p,r}(y) = \left( r -\vert y-p \vert\right)^2 \vert \nabla \Phi(y) \vert^2 \leq C \int_{\mathbb{D}} \vert \nabla \Phi \vert^2  $$
Letting $p \to 0$ and $r\to 1$ we obtain the result. By conformal invariance of the harmonic map equation, we can assume without loss of generality that $p=0$, $r=1$ and $\Phi$ is a $\mathcal{C}^\infty$ map until the boundary of $\mathbb{D}$. We then set for $x\in \mathbb{D}$
$$ F(x) = \left(1-\vert x \vert \right)^2 \vert \nabla \Phi(x) \vert^2. $$
Let $x_0 \in \mathbb{D}$ be such that $F(x_0) = \max_{x\in\mathbb{D}} F(x)$. We set $\sigma_0 = \frac{1-\vert x_0 \vert}{2}$ and for $x \in \mathbb{D}_{\sigma_0}(x_0)$,
$$ 4 \sigma_0^2 \vert \nabla \Phi \vert^2(x_0) \geq (1-\vert x \vert)^2 \vert \nabla \Phi (x) \vert^2 \geq (1-\vert x_0 \vert-\vert x-x_0 \vert)\vert \nabla \Phi (x) \vert^2 \geq \sigma_0^2 \vert \nabla \Phi (x) \vert^2$$
implies $\vert \nabla \Phi (x) \vert^2 \leq 4 \vert \nabla \Phi (x_0) \vert^2$ in $\mathbb{D}_{\sigma_0}(x_0)$. We set
$$ \widetilde{\Phi}(z) = \Phi(x_0 + \sigma_0 z) $$
so that $\widetilde{\Phi}$ is harmonic and
$$ \vert \nabla \widetilde{\Phi}(0) \vert^2 = \sigma_0^2 \vert \nabla \Phi(x_0) \vert^2  $$
By the same computation as for $\Phi$,
$$\frac{1}{2} \Delta \vert \nabla\widetilde{\Phi} \vert^2 = - \vert \nabla^2 \widetilde{\Phi} \vert^2 + \frac{\vert \nabla \widetilde{\Phi} \vert_\Lambda^4}{\vert \Lambda \Phi \vert^2} \leq \alpha^3 \vert \nabla \widetilde{\Phi} \vert^4 \leq 4 \alpha^3 \vert \nabla \widetilde{\Phi}(0) \vert^2 \vert \nabla \widetilde{\Phi} \vert^2. $$
We apply Proposition \ref{prop:monotonicity} to $u = \vert \nabla \widetilde{\Phi} \vert^2$ and $c = 8 \alpha^3 \vert \nabla \widetilde{\Phi}(0) \vert^2 $ to obtain the monotonicity formula
$$ \pi \vert \nabla  \widetilde{\Phi}(0) \vert^2 \leq e^{2 \alpha^3 \vert \nabla \widetilde{\Phi}(0) \vert^2 r^2 } \frac{1}{r^2} \int_{\mathbb{D}_r} \vert \nabla  \widetilde{\Phi} \vert^2. $$
We have two cases: if $\vert \nabla  \widetilde{\Phi}(0) \vert^2 \leq 1$, then with $r=1$,
$$ F(x_0) = 4 \vert \nabla \widetilde{\Phi}(0) \vert^2 \leq \frac{4 e^{2\alpha^3}}{\pi} \int_{\mathbb{D}} \vert \nabla  \widetilde{\Phi} \vert^2 \leq \frac{4 e^{2\alpha^3}}{\pi} \int_{\mathbb{D}} \vert \nabla\Phi \vert^2. $$
If $\vert \nabla  \widetilde{\Phi}(0) \vert^2 \geq 1$ then with $r_0 = \vert \nabla  \widetilde{\Phi}(0) \vert^{-1} $,
$$  \vert \nabla \widetilde{\Phi}(0) \vert^2 \leq \frac{ e^{2\alpha^3}}{\pi r_0^2} \int_{\mathbb{D}_{r_0}} \vert \nabla  \widetilde{\Phi} \vert^2 \leq \frac{ e^{2\alpha^3}}{\pi}  \vert \nabla \widetilde{\Phi}(0) \vert^2 \epsilon_\alpha $$
Choosing $\epsilon_\alpha \leq \pi (2e^{2\alpha^3})^{-1}$, we obtain a contradiction.
\end{proof}

The previous theorem is a consequence of the conformal invariance of the harmonic map equation and the following monotonicity formula:
\begin{prop} \label{prop:monotonicity}
Let $u \in \mathcal{C}^2$ and $c>0$. If
$$ \begin{cases}\Delta u = -div\left(\nabla u\right) \leq c u \\ u\geq 0\end{cases} $$
then $\left(e^{\frac{c}{4}r^2} \left( \frac{1}{r^2} \int_{\mathbb{D}_r} u\right)\right)_{r \geq 0}$  is non decreasing.
\end{prop}
\begin{proof} We compute
\begin{align*} -4 \int_{\mathbb{D}_r} u = \int_{\mathbb{D}_r} u \Delta \vert x \vert^2 = & \int_{\mathbb{D}_r}  \vert x \vert^2 \Delta u + \int_{\partial\mathbb{D}_r}  \vert x \vert^2 \partial_\nu u  - \int_{\partial\mathbb{D}_r} u \partial_{\nu} \vert x \vert^2 \\  = & \int_{\mathbb{D}_r} (\vert x\vert^2 -r^2) \Delta u - 2r \int_{\partial\mathbb{D}_r} u \end{align*}
and setting $f(r) = \frac{1}{r^2} \int_{\mathbb{D}_r} u$, we deduce
$$ f'(r) = \frac{1}{2r^3}\left( -4 \int_{\mathbb{D}_r} u + 2r \int_{\partial \mathbb{D}_r} u \right) =  \frac{1}{2r^3} \int_{\mathbb{D}_r} (\vert x\vert^2 -r^2) \Delta u \geq \frac{-cr}{2} f(r) $$
so that $\left(\ln\left( f(r) e^{\frac{c}{4}r^2} \right)\right)' \geq 0$.
\end{proof}

\subsection{Steklov harmonic eigenmaps}

\begin{theo}\label{theomainSteklov}
For any $\alpha>1$, there is $C_\alpha > 0$ and $\eps_\alpha>0$ such that for every $k\in \mathbb{N}$ and $\sigma = (\sigma_1,\cdots,\sigma_k)$ with
$$ \max_{1\leq i \leq k} \sigma_i \leq \alpha \text{ and } \min_{1\leq i \leq k} \sigma_i \geq \alpha^{-1}, $$
such that if $\Phi : (\mathbb{D}^+,[-1,1]) \to \left(co\left(\mathcal{E}_{\sigma}\right),\mathcal{E}_{\sigma} \right)$ is a free boundary harmonic map then
$$ \int_{\mathbb{D}^+} \left\vert \nabla \Phi \right\vert^2 \leq \eps_\alpha \Rightarrow \forall x \in \mathbb{D}_+, \vert \nabla \Phi(x) \vert^2 \leq C_{\alpha}\frac{\int_{\mathbb{D}^+} \left\vert \nabla \Phi \right\vert^2}{\left( 1-\vert x \vert \right)^2}  $$
\end{theo}

The main difficulty in this result is to notice that we do not need the $\eps_\alpha$ smallness assumption on $\int_{[-1,1]\times\{0\}} \vert \nabla \Phi \vert$, contrary to what was suggested in \cite{kkms}, lemma 4.10. For that purpose, we will use a control of a $L^2$ norm of the gradient on the boundary by $L^2$ and $L^4$ norms of the gradient in the interior in Claim \ref{cl:noneedsmallmassboundary}. We will use a symmetrisation procedure first introduced by Scheven \cite{scheven}. Claim \ref{cluniform} follows Lemma 3.1 in \cite{scheven} where we carefully focus on the independence with respect to the number of coordinates. Then, we use in Claim \eqref{cl:noneedsmallmassboundary} a symmetrization which is well adapted to the ellipsoid because it is quite global $\mathbb{R}^k \setminus\{0\} \to \mathbb{R}^k \setminus \{0\}$. The involution we use is modeled on the inversion (if the ellipsoid is a sphere, it is nothing but an inversion). This correspond to generalize the spherical case in \cite{LP} instead of using the very general symmetrization of \cite{scheven,JLZ}.

\begin{cl}{\cite{scheven}, lemma 3.1} \label{cluniform}
for any $\alpha>0$ there is $\eps_{\alpha} >0$  
such that for any $n\in \mathbb{N}$ and $\Phi : \mathbb{D}^+ \to \mathbb{R}^n$ a Euclidean harmonic map such that $\left\vert \Phi \right\vert^2 \geq 1$ on $[-1,1]\times \{0\}$ and such that
$$ \int_{\mathbb{D}_+} \left\vert \nabla \Phi \right\vert^2 \leq \eps_{\alpha}, $$
we have $\left\vert \Phi \right\vert^2 \geq 1 - \alpha $ on $ \mathbb{D}_{\frac{1}{2}}^+$.
\end{cl}

\begin{proof}
Let $x:= (s, 3r ) \in \mathbb{D}_{\frac{1}{2}}^+ $ be such that $r>0$. We set
$$w(z) = \sqrt{\sum_i (\Phi_i(z) - \bar{\Phi}_i)^2}  \text{ where } \bar{\Phi} := \frac{1}{\mathbb{D}_{5r}(s,0)^+} \int_{\mathbb{D}_{5r}(s,0)^+} \Phi(z)dz $$
We let $G_x$ be the fundamental solution of the Laplace operator that satisfies 
$$ \left\vert \nabla G_x(z) \right\vert \leq \frac{C}{\vert x-z \vert}$$
and $\eta \in \mathcal{C}^{\infty}_c(\mathbb{D}_{2r}(x))$ such that $0\leq \eta \leq 1$, $\eta=1$ in $\mathbb{D}_r(x)$ and $\vert\nabla \eta \vert \leq \frac{C}{r}$. Then,
\begin{equation*}
\begin{split}
 w(x)^2 & = - \int_{\mathbb{D}_{2r}(x)} \nabla G_x \nabla\left( w^2 \eta^2 \right) \\
& = -  \int_{\mathbb{D}_{2r}(x)} \nabla G_x \left(\sum_{i}(\Phi_i - \bar{\Phi}_i) \nabla \Phi_i \right) \eta^2 - 2 \int_{\mathbb{D}_{2r}(x)} \nabla G_x w^2 \eta \nabla \eta \\
& \leq \Vert w \Vert_{L^{\infty}(\mathbb{D}_{2r}(x))} \Vert \nabla G_x \Vert_{L^1(\mathbb{D}_{2r}(x))} \Vert \vert \nabla \Phi \vert \Vert_{L^{\infty}(D_{2r}(x))} + \frac{C}{r} \int_{\mathbb{D}_{2r}(x)\setminus \mathbb{D}_r(x)} w^2
\end{split}
\end{equation*}
By a standard Poincar\'e inequality, we have
$$  \int_{\mathbb{D}_{2r}(x)\setminus \mathbb{D}_r(x)} w^2 \leq \sum_{i} \int_{\mathbb{D}_{5r}^+(s,0)} (\Phi_i - \bar{\Phi}_i)^2 \leq K_1 r^2 \sum_i \int_{\mathbb{D}_{5r}^+(s,0)} \vert \nabla \Phi_i \vert^2 \leq Cr^2 \int_{\mathbb{D}_+} \vert \nabla \Phi \vert^2 $$
and by interior regularity estimates on harmonic functions,
$$ \Vert \vert \nabla \Phi \vert \Vert_{L^{\infty}(D_{2r}(x))}^2 \leq \sum_i \Vert \vert \nabla \Phi_i \vert^2 \Vert_{L^{\infty}(D_{2r}(x))} \leq \frac{C}{r^2} \sum_i \int_{\mathbb{D}_{3r}(p)} \vert \nabla \Phi_i \vert^2 \leq \frac{K_2}{r^2} \int_{\mathbb{D}_+} \vert \nabla \Phi \vert^2 $$
and for $z \in \mathbb{D}_{2r}(x)$
$$ w(z) \leq w(x) + \vert w(x) - w(z) \vert \leq w(x) + \vert \Phi(z) - \Phi(x) \vert \leq w(z) + 2r \Vert \vert \nabla \Phi \vert \Vert_{L^{\infty}(\mathbb{D}_{2r}(x))} $$
so that using the two previous inequalities, and taking the supremum with respect to $z$,
$$ \Vert w \Vert_{L^{\infty}(\mathbb{D}_{2r}(x))} \leq w(x) + 2\sqrt{K_2} \sqrt{ \int_{\mathbb{D}_+} \vert \nabla \Phi \vert^2 } $$
and finally
$$ \Vert \nabla G_x \Vert_{L^1(\mathbb{D}_{2r}(x))} \leq C' r $$
so that gathering all the previous inequalities,
$$ w(x)^2 \leq C' \sqrt{K_2} w(x) \sqrt{ \int_{\mathbb{D}_+} \vert \nabla \Phi \vert^2 } + 2C' K_2  \int_{\mathbb{D}_+} \vert \nabla \Phi \vert^2  + Cr  \int_{\mathbb{D}_+} \vert \nabla \Phi \vert^2 $$
and we obtain
$$ w(x)^2 \leq K \int_{\mathbb{D}_+} \vert \nabla \Phi \vert^2   $$
for the constant $K = ( (2C'K_2 + C + \frac{C'^2 K_2}{4})^{\frac{1}{2}} + \frac{C' \sqrt{K_2}}{2} )^2$.
We then have by a triangle inequality that
$$ 1 - \vert \Phi(x) \vert \leq 1 - \vert \bar \Phi \vert + w(x) $$
where for all $z \in \mathbb{D}_{5r}^+(s,0)$, another triangle inequality gives
$$ 1 - \vert \bar \Phi \vert \leq 1 - \vert \Phi(z) \vert + w(z) \leq (1 - \vert \Phi(z)\vert)_+ + w(z) $$
Knowing by assumption that $[s-5r,s+5r]\times \{0\}$, $(1 - \vert \Phi(z)\vert)_+ = 0$, the Poincaré inequality gives that
$$ \left( \frac{1}{\mathbb{D}_{5r}^+(s,0)} \int_{\mathbb{D}_{5r}^+(s,0)} (1 - \vert \Phi(z)\vert)_+ \right)^2 \leq C \int_{\mathbb{D}_{5r}^+(s,0)} \vert\nabla (1 - \vert \Phi(z)\vert)_+ \vert^2 \leq C \int_{\mathbb{D}_{5r}^+(s,0)} \vert \nabla \Phi(z) \vert^2 $$
since $x \in \R^k \mapsto (1-\vert x \vert)_+ $ is $1$-Lipschitz (independently of $k$) and
$$ \left( \frac{1}{\mathbb{D}_{5r}^+(s,0)} \int_{\mathbb{D}_{5r}^+(s,0)} w(z) \right)^2 \leq C  \int_{\mathbb{D}_{5r}^+(s,0)} \vert \nabla \Phi(z) \vert^2 $$
so that gathering all the previous inequalities,
$$ 1 - \vert \Phi(x) \vert \leq (2\sqrt{C} + \sqrt{K}) \left( \int_{\mathbb{D}^+} \vert \nabla \Phi(z) \vert^2 \right)^{\frac{1}{2}}  $$
Choosing $\eps_\alpha = \frac{\alpha^2}{(2\sqrt{C}+\sqrt{K})^2} $ ends the proof of the Claim.
\end{proof}

\begin{cl} \label{cl:noneedsmallmassboundary}
For any $\alpha>1$, there is $C_\alpha > 0$ and $\eps_\alpha>0$ such that for every $k\in \mathbb{N}$ and $\sigma = (\sigma_1,\cdots,\sigma_k)$ with
$$ \max_{1\leq i \leq k} \sigma_i \leq \alpha \text{ and } \min_{1\leq i \leq k} \sigma_i \geq \alpha^{-1}, $$
such that $u : (\mathbb{D}^+,[-1,1]) \to \left(co\left(\mathcal{E}_{\sigma}\right),\mathcal{E}_{\sigma} \right)$ is a free boundary harmonic map satisfying
$$ \int_{\mathbb{D}^+} \left\vert \nabla u \right\vert^2 \leq \eps_\alpha $$
Then
$$ \int_{[-\frac{9}{10},\frac{9}{10}]\times \{0\}} \vert \nabla u \vert^2 \leq C_{\alpha} \left( \int_{\mathbb{D}^+} \left\vert \nabla u \right\vert^2 +  \int_{\mathbb{D}^+} \left\vert \nabla u \right\vert^4  \right) $$
\end{cl}

\begin{proof} 
\noindent \textbf{Step 1: } We extend $u : \mathbb{D}^+ \to co(\mathcal{E}_\sigma)$ to a map $\hat{u} : \mathbb{D} \to \mathbb{R}^k$ by a symmetrization procedure. We let $\eps > 0$ be such that 
$$ \int_{\mathbb{D}_+} \left\vert \sqrt{\sigma} \nabla u \right\vert^2 \leq \eps $$
implies that $\left\vert \sqrt{\sigma} \Phi \right\vert^2 \geq \frac{1}{2}$ and we will choose $\eps_\alpha \leq \frac{\eps}{\alpha}$. We set for $x \in \mathbb{R}^k\setminus\{0\}$
$$ p(x) = e^{\sigma t(x)} x \text{ and } s(x) = e^{2\sigma t(x)} x $$
where $t(x)$ is the unique real value such that $\sum_{i=1}^k \sigma_i e^{2\sigma_i t(x)} x_i^2 = 1  $. Then, 
$$s\circ s = id_{\R^k\setminus \{0\}} \text{ , } p\circ p = p \text{ , } p\circ s = p \text{ and }  t\circ s + t = 0.$$
A direct computation gives that $\partial_j s_i(x)$ is a symmetric matrix and
$$ Ds(x) = e^{2\sigma t(x)} - 2 \frac{ \sigma s(x) \left( \sigma s(x)\right)^T}{ \sigma s(x) \cdot \sigma x } . $$
Notice also that 
$$ \forall x \in \mathcal{E}_{\sigma}, p(x) = x \text{ and } Dp(x).v = v - \frac{v\cdot \sigma x}{\left\vert \sigma x \right\vert^2} \sigma x $$
and
$$\forall x \in \mathcal{E}_{\sigma}, s(x) = x \text{ and } Ds(x).v = v - 2\frac{v\cdot \sigma x}{\left\vert \sigma x \right\vert^2} \sigma x $$
We set
$$ \hat{u} := \begin{cases} u & \text{ in } \mathbb{D}_+ \\
s\circ u \circ \rho & \text{ in } \mathbb{D}_-
\end{cases} $$
where $\rho(x,-y) = -\rho(x,y)$. Then, we set
$$ A(x) := \begin{cases} Id_{\mathbb{R}^k} & \text{ in } \mathbb{D}_+ \\
\sqrt{S^2} & \text{ in } \mathbb{D}_-
\end{cases} $$
where 
$$ S(x) := Ds(\hat{u}(x)) $$
is an invertible symmetric matrix with inverse matrix $ Ds(u(\rho(x))) = Ds(s( \hat{u}(x))) $ coming from the equality $s\circ s = id_{\R^k\setminus \{0\}}$ and $ M \in \mathcal{S}_{++}^k(\R) \mapsto \sqrt{M} \in \mathcal{S}_{++}^k(\R) $ is the classical smooth map defined on the open set subset of positive definite matrices $\mathcal{S}_{++}^k(\R) \subset \mathcal{M}_k(\R)$. It is important to notice that 
$$ A \in L^{\infty} \cap W^{1,2} \text{ and } A^{-1} \in L^{\infty} \cap W^{1,2} $$
and even that $A \in W^{1,\infty}$ if we a priori assume that $u$ is smooth. Notice first that
$$ -div\left( A \nabla \hat{u} \right) = 0 \text{ in } \mathbb{D}_+.  $$
Notice also that since $\rho$ is an isometry of $\R^2$, and $S(x) \cdot Ds( u(\rho(x))) = id_{\R^k}$ we obtain that
$$ \Delta(u\circ\rho)  = 0 \text{ in } \mathbb{D}_- \text{ and } \nabla{\hat{u}} = S^{-1} \nabla( u\circ \rho) $$
so that in $\mathbb{D}_-$,
$$ -div(A \nabla \hat u) = -div( AS^{-1} \nabla (u\circ \rho) ) = - \nabla( AS^{-1} ) \nabla (u\circ \rho) = - \nabla(AS^{-1}) S A^{-1} A \nabla{\hat{u}} $$
and setting $P := AS^{-1}$, we have that 
$$P^{T} P = (S^{-1})^T A^T A S^{-1} = S^{-1} A^2 S^{-1} = S^{-1} S^2 S^{-1} = id_{\R^k} $$
so that 
$$ \Omega := \begin{cases} 0 & \text{ in } \mathbb{D}_+ \\
P \cdot \nabla P^{T} & \text{ in } \mathbb{D}_-
\end{cases} $$
is antisymmetric and
\begin{equation}  \label{eq:weakequationonD} -div(A \nabla \hat u) = \Omega A \nabla \hat u \text{ in } \mathbb{D}_+ \text{ and in } \mathbb{D}_-. \end{equation}
Let's prove that the equation is satisfied weakly in $\mathbb{D}$. We first aim at proving that 
\begin{equation} \label{eq:firstequationsymmetrized} \forall v \in \mathcal{C}^{\infty}_c\left(\mathbb{D}, \mathbb{R}^k \right),
 \int_{\mathbb{D}_+} \nabla \hat{u} \nabla v + \int_{\mathbb{D}_-} S \nabla \hat{u} \nabla\left( S v \right) = 0 \end{equation}
Setting 
$$ v_e := \frac{1}{2}\left( v + S\circ\rho \cdot v  \circ \rho \right) $$
$$ v_a := \frac{1}{2}\left( v - S\circ\rho \cdot v \circ \rho \right) $$
we obtain knowing that $S v_e = v_e \circ \rho$
\begin{equation*}
\begin{split}
\int_{\mathbb{D}_+} \nabla \hat{u} \nabla v_e   + \int_{\mathbb{D}} S \nabla \hat{u} \nabla\left( S v_e \right) & = \int_{\mathbb{D}_+} \nabla \hat{u} \nabla v_e  + \int_{\mathbb{D}_-} Ds(\hat{u}) \nabla \hat{u} \nabla\left( Ds(\hat{u}) \cdot v_e  \right) \\
& = \int_{\mathbb{D}_+} \nabla \hat{u} \nabla v_e  + \int_{\mathbb{D}_-} \nabla\left( u \circ \rho \right) \nabla\left( v_e \circ \rho \right) \\
& = 2 \int_{\mathbb{D}_+} \nabla u \nabla v_e \\
& = 2 \int_{[-1,1]\times \{0\}} \left(-\partial_y u\right)\cdot  v_e = 0.
\end{split} 
 \end{equation*}
The last term is equal to $0$ since  by construction, $v_e . \sigma u = 0 $ on $[-1,1]\times \{0\}$, and $u$ is free boundary harmonic in the ellipsoid: $-\partial_y u \parallel \sigma u$. Moreover, knowing that $Av_a = - v_a\circ \rho$,
\begin{equation*}
\begin{split}
 \int_{\mathbb{D}_+} \nabla \hat{u} \nabla v_a + \int_{\mathbb{D}_-} S \nabla \hat{u} \nabla\left( S v_a \right) & = \int_{\mathbb{D}_+} \nabla \hat{u} \nabla v_a  + \int_{\mathbb{D}_-} Ds(\hat{u}) \nabla \hat{u} \nabla\left( Ds(\hat{u}) \cdot v_a  \right) \\
& = \int_{\mathbb{D}_+} \nabla \hat{u} \nabla v_a  -  \int_{\mathbb{D}_-} \nabla\left( u \circ \rho \right) \nabla\left( v_a \circ \rho \right) = 0.
\end{split} 
 \end{equation*}
 Therefore, the equation \eqref{eq:firstequationsymmetrized} is satisfied, and by a direct computation, since $S$ and $A$ are symmetric matrices and $S^2 = A^2$, we have on $\mathbb{D}_-$ that
\begin{equation*} \begin{split} S \nabla \hat{u} \cdot \nabla(Sv) & = S \nabla \hat{u} \cdot \nabla(SA^{-1}) A v + S \nabla \hat{u} \cdot SA^{-1} \nabla( A v )  \\
&= \nabla( (SA^{-1})^T ) S \nabla \hat{u} \cdot A v +  A^{-1}S^2 (\nabla \hat{u}) \cdot \nabla(A v) \\
& = (\nabla P ) P^{T} A \nabla \hat{u} \cdot A v +  A (\nabla \hat{u}) \cdot \nabla(A v)
 \end{split} \end{equation*}
and the equation \eqref{eq:weakequationonD} is satisfied weakly in $\mathbb{D}$ since $A$ and $A^{-1}$ are $W^{1,2} \cap L^{\infty}$ matrices. This equation can also be written
\begin{equation} \label{eq:strongeqonhatu} \Delta \hat{u} = \left(A^{-1} \nabla A + A^{-1 }\Omega A \right) \nabla \hat{u} . \end{equation}
Now, we have that in $\mathbb{D}_-$
\begin{equation*} \begin{split} A^{-1} \nabla A + A^{-1 }\Omega A & = A^{-1} \nabla A + A^{-1} A S^{-1} \nabla(S^{-1}A) A \\
& = S^{-1}\left( SA^{-1} \nabla A + \nabla(S^{-1}A) A \right) \\
& = S^{-1} \nabla S  \end{split}  \end{equation*}
where we used again that $A^2 = S^2$ to obtain $S^{-1}A = S A^{-1}$. We finally obtain
\begin{equation}
\label{eq:finalequationhatu}
\Delta \hat{u} = 
\begin{cases} 0 & \text{ in } \mathbb{D}_+ \\
-\nabla( S^{-1} ) \nabla( u \circ \rho) & \text{ in } \mathbb{D}_-
\end{cases} = \begin{cases} 0 & \text{ in } \mathbb{D}_+ \\
-D^2 s(u\circ \rho)\left(\nabla( u \circ \rho), \nabla( u \circ \rho)\right) & \text{ in } \mathbb{D}_-
\end{cases} .
\end{equation}

\medskip

\noindent \textbf{Step 2} There is a constant $K(\alpha) \geq 1$ such that 
$$ K(\alpha)^{-1} \left\vert \nabla u \right\vert^2 \leq  \left\vert \nabla s\circ u \right\vert^2 \leq K(\alpha) \left\vert \nabla u \right\vert^2 \text{ in } \mathbb{D}_+ $$
$$ \sum_{i=1}^k \left\vert \Delta \hat{u}_i \right\vert^2 \leq K(\alpha)^3 \left\vert \nabla \hat{u} \right\vert^4 \text{ in } \mathbb{D}$$

\medskip

\noindent \textbf{Proof of Step 2}

We prove the first formula : notice that by the choice of $t$, we must have that $t\circ u \geq 0$ and
$$ 1 = \sum_{i=1}^k \sigma_i e^{2\sigma_i t(u)}u_i^2 \geq e^{\frac{2}{\alpha}t(u)} \sum_{i=1}^k \sigma_i u_i^2  \geq \frac{1}{2} e^{\frac{2}{\alpha}t(u) } $$
so that for any $i \in \{1,\cdots,k\}$,
$$ e^{2\sigma_i t(u)} \leq e^{2\alpha t(u)} \leq \left( e^{\frac{2}{\alpha} t(u)}\right)^{\alpha^2} \leq 2^{\alpha^2} $$
so that
$$ \nabla ( s_i \circ u) = \sum_j \partial_j s_i(u) \nabla u_j = e^{2\sigma_i t(u)}\nabla u_i - \sum_j \frac{2 \sigma_i e^{2\sigma_i t(u)} u_i \sigma_j e^{2\sigma_j t(u)}u_j \nabla u_j}{\sum_{l} \sigma_l^2 u_l^2 e^{2\sigma_l t(u)}}  $$
leads to
\begin{equation*}
\begin{split} \left\vert  ( s_i \circ u) \right\vert^2 \leq & 2 e^{4\sigma_i t(u)}  \left\vert \nabla u_i  \right\vert^2 + 8 \sigma_i^2 e^{4\sigma_i t(u)} u_i^2 \frac{ \left\vert \sum_j \sigma_j e^{2\sigma_j t(u)}u_j \nabla u_j \right\vert^2}{\left(\sum_{l} \sigma_l^2 u_l^2 e^{2\sigma_l t(u)}\right)^2} \\
\leq & 2^{2\alpha^2 +1}  \left\vert \nabla u_i  \right\vert^2 + 8\sigma_i^2 u_i^2 e^{4\sigma_i t(u)} \frac{\sum_j \sigma_j^2 e^{4\sigma_j t(u)} u_j^2}{\left(\sum_{l} \sigma_l^2 u_l^2 e^{2\sigma_l t(u)}\right)^2} \left\vert \nabla u \right\vert^2 
\end{split}
\end{equation*}
and to
$$ \vert \nabla s\circ u \vert^2 \leq 5 \cdot 2^{2\alpha^2+1} \left\vert \nabla u \right\vert^2 . $$
We also have that
$$ \nabla u = \nabla\left( s\circ s \circ u\right) = Ds(s(u)). \nabla\left( s\circ u\right) $$
so that since $t(s(u)) = -s(u)$,
$$  \nabla u_i = \sum_j \partial_j s_i(s(u)) \nabla( s_j \circ u) = e^{- 2\sigma_i t(u)}\nabla( s_i\circ u) - \sum_j \frac{2 \sigma_i  u_i \sigma_j u_j \nabla (s_j\circ u)}{\sum_{l} \sigma_l^2 u_l^2 e^{2\sigma_l t(u)}} $$
we obtain by straightforward computations since $e^{2\sigma_i t(u)} \geq 1$
$$ \left\vert \nabla u \right\vert^2 \leq 5 \left\vert \nabla s \circ u \right\vert^2  $$
Now, we prove the second formula. We have that
\begin{equation*} \begin{split} \sum_{j,l} \partial_{j,l} s_i(u) = & -2 \frac{\sigma_i \sigma_l e^{2 \sigma_i t(u)}s_l(u) \delta_{i,j} + \sigma_j \sigma_i e^{2 \sigma_j t(u)}s_i(u)\delta_{j,l} + \sigma_l \sigma_j e^{2 \sigma_l t(u)}s_j(u)\delta_{l,i}}{\sum_{m} \sigma_m^2 u_m^2 e^{2\sigma_m t(u)}} \\
& + 4 \frac{\left(\sigma_i^2 \sigma_j \sigma_l + \sigma_j^2 \sigma_l \sigma_i + \sigma_l^2 \sigma_i \sigma_j + \frac{\sum_{m} \sigma_m^3 u_m^2 e^{2\sigma_m t(u)}}{\sum_{m} \sigma_m^2 u_m^2 e^{2\sigma_m t(u)}} \sigma_i \sigma_j \sigma_l \right) s_i(x)s_j(x)s_l(x)}{\left(\sum_{m} \sigma_m^2 u_m^2 e^{2\sigma_m t(u)}\right)^2} \end{split}
 \end{equation*}
\begin{equation*}
\begin{split} \sum_{j,l} \partial_{j,l} s_i(u) \nabla u_j \nabla u_l = & -\frac{2}{\omega^2}\left( \sum_j \sigma_j s_j(u) \nabla u_j \right) \left( 2 \sigma_i e^{2\sigma_i t(u)} \nabla u_i \right) \\
& -\frac{2}{\omega^2} \left(  \sigma_i s_i(u)   \left( \sum_j e^{2\sigma_j t(u)} \left\vert \nabla u_j \right\vert^2 \right) \right) \\
& + \frac{8}{\omega^4} \left( \sigma_i s_i(u) \left( \sum_j \sigma_j s_j(u) \nabla u_j \right)\left( \sum_j \sigma_j^2 s_j(u) \nabla u_j \right)  \right) \\
& + \frac{4}{\omega^4}  \left( \sigma_i + \frac{\tilde{\omega}^2}{\omega^2} \right) \sigma_i s_i(u) \left( \sum_j \sigma_j s_k(u) \nabla u_j \right)^2 
\end{split} \end{equation*}
where $\omega^2 := \sum_{m} \sigma_m^2 u_m^2 e^{2\sigma_m t(u)}$ and $\tilde{\omega}^2:= \sum_{m} \sigma_m^3 u_m^2 e^{2\sigma_m t(u)}$ and
\begin{equation*}
\begin{split} \left(\sum_{j,l} \partial_{j,l} s_i(u) \nabla u_j \nabla u_l\right)^2 \leq & \frac{16}{\omega^4}\left( \sum_j \sigma_j^2 e^{4\sigma_j t} u_j^2 \right) \left\vert \nabla u \right\vert^2 4 \sigma_i^2 e^{4\sigma_i t(u)} \left\vert \nabla u_i \right\vert^2  \\
& + \frac{16}{\omega^4} \left(  \sigma_i^2 e^{4\sigma_i t} u_i^2   \left( \sum_j e^{2\sigma_j t(u)} \left\vert \nabla u_j \right\vert^2 \right)^2 \right) \\
& + \frac{256}{\omega^8} \left( \sigma_i^2 e^{4\sigma_i t} u_i^2  \left( \sum_j \sigma_j^2 e^{4\sigma_j t} u_j^2 \right)\left( \sum_j \sigma_j^4 e^{4\sigma_j t} u_j^2 \right) \left\vert \nabla u \right\vert^4  \right) \\
& + \frac{32}{\omega^8}   \left(\sigma_i^4 + \sigma_i^2  \frac{\tilde{\omega}^4}{\omega^4}  \right) e^{4\sigma_i t} u_i^2\left( \sum_j \sigma_j^2 e^{4\sigma_j t} u_j^2 \right)^2 \left\vert \nabla u \right\vert^4
\end{split} \end{equation*}
and finally
$$ \sum_i \left(\sum_{j,l} \partial_{j,l} s_i(u) \nabla u_j \nabla u_l\right)^2 \leq \left(64 \cdot 2^{2\alpha^2} +  16 \cdot 2^{3\alpha^2} + 256\cdot \alpha^2 2^{3\alpha^2} + 64 \alpha^2 2^{3\alpha^2} \right) \left\vert \nabla u \right\vert^4 . $$
From all these inequalities, up to increase $K(\alpha)$, we obtain Step 2.

\medskip

\noindent \textbf{Step 3: conclusion}

\medskip

By classical trace embedding theorems, we obtain that
$$ \int_{[-\frac{9}{10},\frac{9}{10}]\times\{0\}} \left\vert \nabla \hat{u}_i \right\vert^2 \leq C \left( \int_{\mathbb{D}_{\frac{9}{10}}} \left\vert \nabla^2 \hat{u}_i \right\vert^2 + \int_{\mathbb{D}_{\frac{9}{10}}} \left\vert \nabla \hat{u}_i \right\vert^2 \right) $$
so that by a classical use of Bochner's identity and cut-off functions, we obtain that
$$ \int_{[-\frac{9}{10},\frac{9}{10}]\times\{0\}} \left\vert \nabla \hat{u}_i \right\vert^2 \leq C' \left( \int_{\mathbb{D}} \left\vert \Delta \hat{u}_i \right\vert^2 + \int_{\mathbb{D}} \left\vert \nabla \hat{u}_i \right\vert^2 \right) $$
and summing on $i$,
$$  \int_{[-\frac{9}{10},\frac{9}{10}]\times\{0\}} \left\vert \nabla \hat{u} \right\vert^2 \leq C(\alpha) \int_{\mathbb{D}} \left( \left\vert \nabla \hat{u} \right\vert^4 +  \left\vert \nabla \hat{u} \right\vert^2 \right) $$
Going back to $u$, we then obtain
$$  \int_{[-\frac{9}{10},\frac{9}{10}]\times\{0\}} \left\vert \nabla u \right\vert^2 \leq C'(\alpha) \int_{\mathbb{D}_+} \left( \left\vert \nabla u \right\vert^4 +  \left\vert \nabla u \right\vert^2 \right). $$
\end{proof}

\begin{proof}[Proof of Theorem \ref{theomainSteklov}]
We set $F(x) = \left( 1 - \left\vert x \right\vert \right)^2 \left\vert \nabla \Phi   \right\vert^2(x)$ and we let $x_0 \in \mathbb{D}^+$ be such that $F(x_0) = \sup_{x\in \mathbb{D}_{r_0}(p)} F(x) $. It suffices to get a constant $C(\alpha)>0$ such that
$$ F(x_0) =  \left( 1 - \left\vert x_0 \right\vert \right)^2 \left\vert \nabla \Phi   \right\vert^2(x_0) \leq C(\alpha) \int_{\mathbb{D}_+} \vert \nabla \Phi \vert^2. $$
We let $\sigma_0 := \frac{1-\vert x_0 \vert}{2}$ and $\delta_0 := dist(x_0,[-1,1]\times\{0\})$. Then interior suharmonicity implies
\begin{equation} \label{eq:subharmonicityinterior} \pi \delta_0^2 \vert \nabla \Phi \vert^2(x_0) \leq \int_{\mathbb{D}_{\delta_0}} \vert \nabla \Phi \vert^2. \end{equation} 
Notice first that if $\sigma_0 \leq 10 \delta_0$, this gives the expected result
so that without loss of generality, we can assume that $\delta_0 \leq \frac{\sigma_0}{10} $. For $x \in \mathbb{D}_{\sigma_0}(x_0)$, we have that
$$ 4\sigma_0^2 \vert \nabla \Phi \vert^2(x_0) \geq (1-\vert x \vert)^2 \vert \nabla \Phi \vert^2(x) \geq (1-\vert x_0 \vert - \vert x-x_0 \vert)^2 \vert \nabla \Phi \vert^2(x) \geq \sigma_0^2 \vert \nabla \Phi \vert^2(x) $$
so that
$$\vert \nabla \Phi \vert^2(x) \leq 4 \vert \nabla \Phi \vert^2(x_0) \text{ in }  \mathbb{D}_{\sigma_0}(x_0). $$
Now, we compute the equation on $\vert \nabla \Phi \vert^2$ in $\mathbb{D}_{\sigma_0}(x_0)$. Since $\Phi$ is harmonic, $\vert \nabla \Phi \vert^2$ is subharmonic and using $\Delta \Phi = 0$, $\partial_y \Phi \parallel \sigma\Phi$ on $[-1,1]\times \{0\}$,
$$ - \frac{1}{2} \partial_y \vert \nabla \Phi \vert^2 = - \partial_y \Phi \cdot \partial_{yy} \Phi - \partial_x \Phi \cdot \partial_{xy} \Phi =\frac{\vert\partial_y \Phi\vert}{\vert \sigma \Phi \vert} \sigma\Phi \cdot \partial_{xx} \Phi + \partial_x \Phi \cdot \partial_x\left( \frac{\vert\partial_y \Phi\vert}{\vert \sigma \Phi \vert} \sigma\Phi \right)$$
and knowing that $\partial_x \Phi \perp \sigma\Phi$ on $[-1,1]\times \{0\}$,
$$ - \frac{1}{2} \partial_y \vert \nabla \Phi \vert^2 = 2 \frac{\vert\partial_y \Phi\vert}{\vert \sigma \Phi \vert} \vert\partial_x \Phi \vert^2_\sigma . $$
Setting $ x_0 = (p_0 , \delta_0)$ and $y_0 = (p_0,0)$ and taking $s \leq \frac{9\sigma_0}{10}$, we have that
\begin{equation*} \begin{split} \frac{d}{ds} \left( \frac{1}{s} \int_{\left(\partial \mathbb{D}_s(y_0) \right)^+ } \vert \nabla \Phi \vert^2 \right) & = \frac{1}{s} \int_{\mathbb{D}_s^+(y_0)} - \Delta \vert \nabla \Phi \vert^2 + \frac{1}{s} \int_{[p_0-s,p_0+s]\times\{0\}   } \partial_y \vert \nabla \Phi \vert^2 \\
& \geq  - \frac{4}{s} \int_{[p_0-s,p_0+s]\times\{0\}   } \frac{\vert\partial_y \Phi\vert}{\vert \sigma \Phi \vert} \vert\partial_x \Phi \vert^2_\sigma \\
& \geq -\frac{4}{s} \alpha^2  \int_{[p_0-s,p_0+s]\times\{0\}   } \vert \nabla \Phi \vert^3 \\
& \geq  -\frac{4}{s} \alpha^2 \left( 4 \vert \nabla \Phi \vert(x_0) \right)^{\frac{5}{2}} \left(\int_{[p_0-s,p_0+s]\times\{0\}   } \vert \nabla \Phi \vert^2 \right)^{\frac{1}{4}} (2s)^{\frac{3}{4}}
 \end{split}\end{equation*}
We apply Claim \ref{cl:noneedsmallmassboundary} on the rescaled map $u(x) = \Phi(y_0 + \frac{10}{9}s x)$ to obtain
$$ \int_{[-\frac{9}{10},\frac{9}{10}]\times\{0\}   } \vert \nabla u\vert^2 \leq C_\alpha \left( \int_{\mathbb{D}_+} \vert \nabla u \vert^2 + \int_{\mathbb{D}_+} \vert \nabla u \vert^4 \right) $$
rescaled as
$$ \int_{[p_0- s,p_0+ s]\times\{0\}  } \vert \nabla \Phi \vert^2 \leq \frac{10}{9}C_\alpha \left( \frac{1}{s} \int_{\mathbb{D}_{\frac{10}{9}s}^+(y_0)} \vert \nabla \Phi \vert^2 + s \int_{\mathbb{D}_{\frac{10}{9}s}^+(y_0)} \vert \nabla \Phi \vert^4 \right) $$
and this gives
\begin{equation*}  \begin{split}  \frac{d}{ds} \left( \frac{1}{s} \int_{\left(\partial \mathbb{D}_s(y_0) \right)^+ } \vert \nabla \Phi \vert^2 \right) & \geq - K_\alpha s^{-\frac{1}{2}} \left(\vert \nabla \Phi \vert(x_0)\right)^{\frac{5}{2}} (1 + 4 s^2\left(\vert \nabla \Phi \vert(x_0)\right)^{2}  )^{\frac{1}{4}} \left(\int_{\mathbb{D}_+} \vert \nabla \Phi \vert^2\right)^{\frac{1}{4}}  \\
& \geq - K_\alpha \left( s^{-\frac{1}{2}} \left( \vert \nabla \Phi \vert(x_0)   \right)^{\frac{5}{2}} +   \left( \vert \nabla \Phi \vert(x_0)   \right)^{3} \right) \eps_\alpha^{\frac{1}{4}}
 \end{split}\end{equation*}
for some constant $K_\alpha >0$. Integrating with respect to $[t,s]\subset [0,\frac{\sigma_0}{2}]$ gives
$$ \frac{1}{t} \int_{\left(\partial \mathbb{D}_t(y_0) \right)^+ } \vert \nabla \Phi \vert^2 \leq  \frac{1}{s} \int_{\left(\partial \mathbb{D}_s(y_0) \right)^+ } \vert \nabla \Phi \vert^2 + 2K_\alpha \left( s^{\frac{1}{2}} \left( \vert \nabla \Phi \vert(x_0)   \right)^{\frac{5}{2}} +  s \left( \vert \nabla \Phi \vert(x_0)   \right)^{3} \right) \eps_\alpha^{\frac{1}{4}} $$
Multiplying by s and integrating with respect to $s\in [\frac{\sigma}{2},\sigma] \subset [t,\frac{\sigma_0}{2}]$ leads to
$$ \frac{\sigma^2}{4t} \int_{\left(\partial \mathbb{D}_t(y_0) \right)^+ } \vert \nabla \Phi \vert^2 \leq \int_{ \mathbb{D}_\sigma^+(y_0) } \vert \nabla \Phi \vert^2 + 2K_\alpha \left( \sigma^{\frac{5}{2}} \left( \vert \nabla \Phi \vert(x_0)   \right)^{\frac{5}{2}} +  \sigma^3 \left( \vert \nabla \Phi \vert(x_0)   \right)^{3} \right) \eps_\alpha^{\frac{1}{4}}  $$
Multiplying by $t$ and integrating on $[0,2\delta_0] \subset [0,\frac{\sigma_0}{2}]$, we obtain
$$ \frac{\sigma^2}{4} \int_{ \mathbb{D}_{2\delta_0}(y_0)^+ } \vert \nabla \Phi \vert^2 \leq \frac{\delta_0^2}{2} \int_{\mathbb{D}_+}\vert \nabla\Phi \vert^2 + \delta_0^2 K_\alpha \left( \sigma^{\frac{5}{2}} \left( \vert \nabla \Phi \vert(x_0)   \right)^{\frac{5}{2}} +  \sigma^3 \left( \vert \nabla \Phi \vert(x_0)   \right)^{3} \right) \eps_\alpha^{\frac{1}{4}}  $$
Knowing that $\vert \nabla \Phi \vert^2(x_0) \leq \frac{1}{\pi\delta_0^2} \int_{\mathbb{D}_{\delta_0}(x_0)} \vert \nabla \Phi \vert^2 \leq \frac{1}{\pi\delta_0^2} \int_{\mathbb{D}_{2\delta_0}^+(x_0)} \vert \nabla \Phi \vert^2  $ we obtain that 
$$ \left( \sigma \vert \nabla \Phi \vert(x_0) \right)^2 \leq  \int_{\mathbb{D}_+}\vert \nabla\Phi \vert^2 + K_\alpha \left( \sigma^{\frac{5}{2}} \left( \vert \nabla \Phi \vert(x_0)   \right)^{\frac{5}{2}} +   \sigma^3 \left( \vert \nabla \Phi \vert(x_0)   \right)^{3} \right) \eps_\alpha^{\frac{1}{4}} $$
for all $\sigma \in [2\delta_0,\frac{\sigma_0}{2}]$. By \eqref{eq:subharmonicityinterior}, it also holds true for $\sigma \in [0,\frac{\sigma_0}{2}]$. We set
$$ \sigma^2 := 2 \frac{ \int_{\mathbb{D}_+}\vert \nabla\Phi \vert^2}{ \vert \nabla\Phi \vert^2(x_0)  }. $$
If $\sigma \geq \frac{\sigma_0}{2}$ and the expected estimate is proved. We assume that $\sigma \leq \frac{\sigma_0}{2}$ and we obtain
$$ 2  \int_{\mathbb{D}_+}\vert \nabla\Phi \vert^2 \leq  \int_{\mathbb{D}_+}\vert \nabla\Phi \vert^2 + K_\alpha \eps_\alpha^{\frac{1}{4}} \left( (2\eps_\alpha)^{\frac{5}{4}} + (2\eps_\alpha)^{\frac{3}{2}} \right) $$
so that choosing $\eps_\alpha \leq 1$ such that $8 K_\alpha \eps_\alpha^{\frac{1}{2}} \leq 1  $ leads to a contradiction.
\end{proof}

\section{Another assumption for $\eps$-regularity results on eigenmaps}
We aim at proving Theorem \ref{theo:Linftyestimateofgradpsi}. It is a consequence of Proposition \ref{propLinftyestimateofgradpsi}. Indeed, we prove \eqref{eqconsequenceclaimalphaplus} with \eqref{eqconsequenceclaimalpha} at $x \in \mathbb{D}$ by conformal invariance of the harmonic map equation and of the Dirichlet energy:
$$(1-\vert x \vert)^2 \vert \nabla \Psi(x)\vert^2 \leq (1-\vert x \vert)^2 \Vert \nabla \Psi \Vert_{L^{\infty}\left(\mathbb{D}_{\frac{1-\vert x \vert}{4}}(x)\right)}^2 \leq C_2 \int_{\mathbb{D}_{1-\vert x \vert}(x)} \vert \nabla \Psi \vert^2 \leq C \int_{\mathbb{D}} \vert \nabla \Psi \vert^2.$$

\subsection{Case of a general Rivière system}

\begin{prop} \label{propriviereuhlenbeck}
There is $\eps_0>0$ such that for any $1 \leq p < +\infty$ and $1\leq q < 2$, there is $C_p >0 $ and $D_q>0$ such that for any $m \geq 2$ and for any map $\Psi: \mathbb{D} \to \mathbb{R}^{m}$ such that
$$ div \nabla \Psi = \Omega . \nabla \Psi  $$ 
where $\Omega_{i,j} = -\Omega_{i,j} : \mathbb{D} \to \R^2 $ are measurable functions such that
$$ \int_{\mathbb{D}} \left\vert \nabla \Psi \right\vert^2 < +\infty \text{ and } \int_{\mathbb{D}} \left\vert \Omega\right\vert^2 \leq \eps_0 $$
then we have
\begin{equation} \label{eq:Lpestimate} \sum_{i=1}^m \Vert \nabla \psi_i  \Vert_{L^{2p}(\mathbb{D}_{\frac{1}{2}})}^2 \leq C_p \| \nabla \Psi \|_{L^{2}\left(\mathbb{D}\right)}^2. \end{equation}
and
\begin{equation} \label{eq:Lqestimate} \sum_{i=1}^m \Vert \Delta \psi_i  \Vert_{L^q(\mathbb{D}_{\frac{1}{2}})}^2 \leq C_q \| \Omega \|_{L^{2}\left(\mathbb{D}\right)}^2 \| \nabla \Psi \|_{L^{2}\left(\mathbb{D}\right)}^2. \end{equation}
\end{prop}

From now on, for a matrix map $M : \mathbb{D} \to \mathcal{M}_{m}(\mathbb{R})$ we choose the following $L^2$ norm
$$\left\|  M  \right\|_{L^2} =  \left( \int_\mathbb{D} \left\vert M \right\vert^2 \right)^{\frac{1}{2}}  \text{ and } \left\vert M \right\vert^2 = \sum_{i,j} M_{i,j}^2.$$
Notice that we have for any $M,N \in \mathcal{M}_m\left(\R\right)$ and for any $P\in SO(m)$, that $\left\vert M \right\vert = \left\vert PM \right\vert = \left\vert MP \right\vert$ and that $\left\vert M N \right\vert = \left\vert N M \right\vert \leq  \left\vert N \right\vert \left\vert M \right\vert $.

We first use the following construction of $L^2$ Uhlenbeck gauge with variational methods by Schikorra that gives estimates independent of the number of coordinates of $\Psi$.

\begin{prop}[\cite{schikorra}] \label{propgauge}
Let $\Omega \in L^2(\mathbb{D},so(m))$. Then there is $P \in W^{1,2}\left( \mathbb{D} , SO(m) \right)$  and $\xi \in W^{1,2}_0\left(\mathbb{D}, so(m)\right) $. such that
$$ P^T\Omega P = P^T \nabla P + \nabla^\perp \xi$$
and 
$$ \left\|  \nabla P \right\|_{L^2} + \left\|  \nabla \xi \right\|_{L^2} \leq 3   \left\|  \Omega \right\|_{L^2}.$$
\end{prop}

\begin{proof} The proof can be written exactly as in \cite{schikorra}, Theorem 2.1. We emphasize that the computations of the Euler-Lagrange equation in \cite{schikorra}, Lemma 2.4 are written until the boundary: For all $A \in so(m)$ and any $\varphi \in \mathcal{C}^\infty(\overline{\mathbb{D}})$,
$$ \int_{\mathbb{D}} \langle P^T\Omega P - P^T \nabla P, A \rangle \nabla \varphi = 0 $$
that is
$$\int_{\partial \mathbb{D}} \langle P^T\Omega P - P^T \nabla P, A \rangle \cdot \nu \varphi  - \int_{\mathbb{D}} div\left( \langle P^T\Omega P - P^T \nabla P, A \rangle  \right) \varphi = 0 $$
We obtain that $P^T\Omega P = P^T \nabla P + \nabla^\perp \xi$ for $\xi \in W^{1,2}(\mathbb{D})$ and the boundary condition yields $\partial_\tau \xi = 0$. Up to renormalizing by a constant, $\xi = 0$ on $\partial\mathbb{D}$.
\end{proof}

\begin{prop}[\cite{riviere}, Theorem I.4, \cite{daliopalmurella} Theorem 1.3, \text{ and } Proposition \ref{propgauge}] \label{propriviereform}
There is a constant $K$ and a constant $\eps_0>0$ such that for any $m \in \mathbb{N}^\star$, a solution $\Psi: \mathbb{D} \to \mathbb{R}^{m}$ of 
$$ div \nabla \Psi = \Omega . \nabla \Psi  $$ 
where $\Omega \in L^2(\mathbb{D},so(m))$ such that
$$ \int_{\mathbb{D}} \vert \Omega \vert^2 \leq \eps_0 $$
can be written
$$ div(A \nabla \Psi) = \nabla^\perp B \nabla \Psi $$
where $A \in W^{1,2}\cap L^{\infty}\left( \mathbb{D}, \mathcal{M}_m(\R) \right)$ and $B \in W^{1,2}_0( \mathbb{D}, \mathcal{M}_m(\R))$ and
$$ \sum_{i,j} \Vert (PA - I_m)_{i,j} \Vert_{L^\infty(\mathbb{D})}^2  + \int_{\mathbb{D}} \left\vert \nabla A \right\vert^2 
+  \int_{\mathbb{D}} \left\vert \nabla B \right\vert^2 \leq K \int_{\mathbb{D}} \left\vert \Omega \right\vert^2$$
where $P$ is given by Proposition \ref{propgauge}
\end{prop}
\begin{proof} In the following proof, $C$ denotes a universal constant that is associated to elliptic estimates and Wente's inequalities.
The space $E = W^{1,2}\cap L^{\infty}\left( \mathbb{D}, \mathcal{M}_m(\R) \right) \times W^{1,2}_0( \mathbb{D}, \mathcal{M}_m(\R))$ is  endowed with the norm
$$\Vert (\alpha,\beta) \Vert^2 = \sum_{i,j} \Vert  \alpha_{i,j}  \Vert_{L^\infty(\mathbb{D})}^2 + \int_{\mathbb{D}}\vert \nabla \alpha \vert^2 + \int_{\mathbb{D}} \vert \nabla \beta \vert^2 $$
and we define for $(\alpha,\beta)\in E$ the operator $T : E \to E$ as $(R,S) = T(\alpha,\beta)$ the solution of
\begin{equation}
\begin{cases}
\Delta R = \nabla^\perp \beta \nabla P + \nabla \alpha \nabla^\perp \xi \text{ in } \mathbb{D} \\
\partial_\nu R = \alpha \partial_\tau \xi \text{ on } \partial\mathbb{D} \text{ and } \int_{\mathbb{D}} R = 0 \\
\Delta S = \nabla \alpha \nabla^\perp P^T - div\left( (\alpha+I_m) \left(\nabla \xi\right) P^T \right) \text{ in  } \mathbb{D} \\
S= 0 \text{ on } \partial\mathbb{D}
 \end{cases}
\end{equation}
defined as $R = U + V $ where $U,V \in W^{1,2}\cap L^{\infty}\left( \mathbb{D}, \mathcal{M}_m(\R) \right)$ are defined as
$$ U_{i,j} = \sum_{k} U_{i,j,k} \text{ and } V_{i,j} = \sum_{k} V_{i,j,k} $$
where $U_{i,j,k}$ and $V_{i,j,k}$ are the unique solutions of
$$ \begin{cases} \Delta U_{i,j,k} = \nabla^\perp \beta_{i,k} \nabla P_{k,j} \text{ in } \mathbb{D} \\ 
\partial_\nu U_{i,j,k} = 0 \text{ on } \partial\mathbb{D} \text{ and } \int_{\mathbb{D}} U_{i,j,k} = 0
\end{cases} \begin{cases} \Delta V_{i,j,k} =  \nabla \alpha_{i,k} \nabla^\perp \xi_{k,j} \text{ in } \mathbb{D} \\ 
\partial_\nu V_{i,j,k} = 0 \text{ on } \partial\mathbb{D} \text{ and } \int_{\mathbb{D}} V_{i,j,k} = 0
\end{cases} $$
and $S = X+Y+Z $ where $X,Y,Z \in W^{1,2}_0\left( \mathbb{D}, \mathcal{M}_m(\R) \right)$ are defined as
$$ X_{i,j} = \sum_k X_{i,j,k}$$
where $X_{i,j,k}$, $Y$ and $Z$ are the unique solutions of
$$ \begin{cases} \Delta X_{i,j,k} = \nabla \alpha_{i,k} \nabla^\perp P_{j,k} \text{ in } \mathbb{D} \\ 
X_{i,j,k} = 0 \text{ on } \partial\mathbb{D}
\end{cases} \begin{cases} \Delta Y =  div\left( \alpha \nabla \xi P^T \right) \text{ in } \mathbb{D} \\ 
Y = 0 \text{ on } \partial\mathbb{D}
\end{cases} \begin{cases} \Delta Z =  div\left( \nabla \xi P^T \right) \text{ in } \mathbb{D} \\ 
Z = 0 \text{ on } \partial\mathbb{D}
\end{cases} $$
A Wente inequality with Neumann boundary conditions applies with the assumptions $\beta,\xi \in W^{1,2}_0\left( \mathbb{D}, \mathcal{M}_m(\R) \right)$ (see \cite{daliopalmurella})
\begin{equation*} \begin{split} & \Vert U_{i,j} \Vert_{L^\infty}^2 + \Vert \nabla U_{i,j} \Vert_{L^2}^2 \leq \left( \Vert U_{i,j} \Vert_{L^\infty} + \Vert \nabla U_{i,j} \Vert_{L^2}\right)^2 \leq \left( \sum_k \left(\Vert U_{i,j,k} \Vert_{L^\infty} + \Vert U_{i,j,k} \Vert_{L^2}\right)  \right)^2 \\ & \leq C \left( \sum_k \Vert \nabla \beta_{i,k} \Vert_{L^2} \Vert \nabla P_{k,j} \Vert_{L^2}\right)^2
\leq \sum_k \Vert \nabla \beta_{i,k} \Vert_{L^2}^2 \sum_k \Vert \nabla P_{k,j} \Vert_{L^2}^2 \end{split} \end{equation*}
so that summing on $(i,j)$, 
$$ \sum_{i,j} \Vert U_{i,j} \Vert_{L^\infty}^2 + \int_\mathbb{D} \vert \nabla U \vert^2 \leq C \int_{\mathbb{D}} \vert \nabla \beta \vert^2 \int_{\mathbb{D}} \vert \nabla P \vert^2 $$
and the same argument gives
$$ \sum_{i,j} \Vert V_{i,j} \Vert_{L^\infty}^2 + \int_\mathbb{D} \vert \nabla V \vert^2 \leq C \int_{\mathbb{D}} \vert \nabla \alpha \vert^2 \int_{\mathbb{D}} \vert \nabla \xi \vert^2 $$
Now the classical Wente inequality (with Dirichlet boundary conditions) see \cite{wente} gives
$$ \Vert \nabla X_{i,j} \Vert^2_{L^2} \leq C  \left( \sum_k \Vert \nabla \alpha_{i,k} \Vert_{L^2} \Vert \nabla P_{j,k} \Vert_{L^2}\right)^2
\leq \sum_k \Vert \nabla \alpha_{i,k} \Vert_{L^2}^2 \sum_k \Vert \nabla P_{j,k} \Vert_{L^2}^2 $$
so that summing on $(i,j)$ gives
$$ \int_{\mathbb{D}} \vert \nabla X \vert^2 \leq C \int_{\mathbb{D}} \vert \nabla \alpha \vert^2 \int_{\mathbb{D}} \vert \nabla P \vert^2 $$
and standard elliptic estimates yield
$$ \Vert \nabla Y_{i,j} \Vert_{L^2}^2  \leq C \left\Vert \sum_{k,l} \alpha_{i,k} \nabla \xi_{k,l} P_{j,l} \right\Vert_{L^2}^2 $$
so that summing on $i,j$,
$$ \int_{\mathbb{D}} \vert \nabla Y \vert^2 \leq C \int_{\mathbb{D}} \vert \alpha (\nabla \xi) P^T \vert^2 = C \int_{\mathbb{D}} \vert \alpha \nabla \xi \vert^2 \leq C  \int_{\mathbb{D}} \vert \alpha \vert^2 \vert \nabla \xi \vert^2 \leq C \Vert \vert \alpha \vert \Vert_{L^\infty}^2 \int_{\mathbb{D}} \vert \nabla \xi \vert^2 $$
A similar argument yields
$$ \int_{\mathbb{D}} \vert \nabla Z \vert^2 \leq C \int_{\mathbb{D}} \vert \nabla \xi \vert^2.  $$
Gathering all the previous inequalities, using that
$$ \Vert \vert \alpha \vert \Vert_{L^\infty}^2 \leq \sum_{i,j} \Vert \alpha_{i,j} \Vert_{L^\infty}^2$$
and using Proposition \ref{propgauge}, up to increasing $C$ we obtain that
\begin{equation} \label{eq:estriviere} \Vert T(\alpha,\beta) \Vert^2 = \Vert (R,S) \Vert^2 = \Vert (U+V,X+Y+Z) \Vert^2 \leq C \left( \int_{\mathbb{D}} \vert \Omega \vert^2 \Vert (\alpha,\beta) \Vert^2 + \int_{\mathbb{D}} \vert \Omega \vert^2 \right) \end{equation}
and if $(\alpha_1,\beta_1), (\alpha_2,\beta_2) \in E$, since $T$ is affine,
$$ \Vert T(\alpha_1,\beta_1)- T(\alpha_2,\beta_2) \Vert^2 \leq C  \int_{\mathbb{D}} \vert \Omega \vert^2 \Vert (\alpha_1,\beta_1)- (\alpha_2,\beta_2) \Vert^2 .$$
Choosing $\eps_0$ so that $C \eps_0 < \frac{1}{4}$, $T$ is a $\frac{1}{2}$-contraction and has a unique fixed point denoted by $(\hat{A},B)$. The first equation given by $T(\hat{A},B) = (\hat{A},B)$ can be written
$$ div \left(  \nabla \hat{A} + \left(\hat{A}+I_m\right) \nabla^\perp \xi - \left(\nabla^\perp B\right)P \right) = 0  $$
so that there is $W$ such that
$$ \nabla \hat{A} + \left(\hat{A}+I_m\right) \nabla^\perp \xi - \left(\nabla^\perp B\right)P = \nabla^\perp W $$
and in particular, using the boundary conditions on $A,\xi,B$, $\partial_\tau W = 0$. Up to renormalizing by a constant, we assume that $W \in W^{1,2}_0(\mathbb{D},\mathcal{M}_m(\R))$ and we deduce from the second equation in $T(\hat{A},B) = (\hat{A},B)$ that
$$ \begin{cases} div(\left(\nabla W\right) P^T) = 0 \text{ in } \mathbb{D} \\
W = 0 \text{ on } \partial \mathbb{D}.
\end{cases} $$
Then, we can write $\left(\nabla W\right) P^T = \nabla^\perp L$ and we have the system
$$ \begin{cases} div\left(\nabla W\right) = \nabla^\perp L \nabla P^T \text{ in } \mathbb{D} \\
W = 0 \text{ on } \partial \mathbb{D}.
\end{cases} $$
that leads to the Wente inequality
$$ \int_\mathbb{D} \vert \nabla W \vert^2 \leq C \int_\mathbb{D} \vert \nabla L \vert^2 \int_\mathbb{D} \vert \nabla P \vert^2 \leq C \eps_0 \int_\mathbb{D} \vert \nabla W \vert^2. $$
We obtain up to reducing $\eps_0$ (the constant $\eps_0$ is still independent of $m$), we obtain that $W = 0$. We now set $A = \left(I_m+ \hat{A}\right)P^T$ so that
$$ \nabla^\perp B = (I_m+ \hat{A}) (\nabla^\perp \xi) P^T + \left(\nabla \hat{A}\right) P^T  \text{ and } div(P^T \nabla \Phi) = \nabla^\perp \xi P^T \nabla \Phi $$
imply
$$ div(A \nabla \Phi ) = \nabla^\perp B \nabla \Phi.$$
The estimate \eqref{eq:estriviere} applied to $T(\hat{A},B) = (\hat{A},B)$ completes the proof of the proposition.
\end{proof}

We are now in position to prove Proposition \ref{propriviereuhlenbeck}. It is based on Hodge decomposition, Wente's inequality, and Morrey's elliptic estimates that are used in \cite{riviere,sharptopping,khomrutaischikorra}.

\begin{proof}[Proof of Proposition \ref{propriviereuhlenbeck}] 

\medskip

\noindent \textbf{Step 1} : We prove
\begin{equation}\label{peqpropriviere} \sup_{p\in \mathbb{D}_{\frac{1}{2}} ; 0<r<\frac{1}{4}}  \int_{\mathbb{D}_r(p) } r^{-2\alpha} \left\vert \nabla \Psi \right\vert^2 \leq C_0 \| \nabla \Psi \|_{L^{2}\left(\mathbb{D}\right)}^2 \end{equation}
so that
\begin{equation}\label{peqpropriviere2}  \sum_{i=1}^m  \sup_{p\in \mathbb{D}_{\frac{1}{2}} ; 0<r<\frac{1}{4}} \left( \int_{\mathbb{D}_r(p) } r^{-\alpha} \left\vert \Delta \psi_i \right\vert \right)^2 \leq C_0 \| \Omega \|_{L^{2}\left(\mathbb{D}\right)}^2 \| \nabla \Psi \|_{L^{2}\left(\mathbb{D}\right)}^2 \end{equation}

\medskip

Starting with the matrix $A$ of Proposition \ref{propriviereform}, we take the following Hodge decomposition
$$A \nabla \Psi =   \nabla D + \nabla^{\perp} E + \nabla^\perp H$$
where $E = 0$ on $\partial \mathbb{D}$, $D = 0$ on 
$\partial \mathbb{D}$ and $H$ is a Euclidean harmonic map. 
Wente inequalities on the equations of the coordinates of $E$ and $D$ give that
$$ \| \nabla E \|_{L^{2}}^2 \leq C \| \nabla A \|_{L^2}^2  \| \nabla \Psi \|_{L^2}^2  $$
$$ \| \nabla D \|_{L^{2}}^2 \leq C \| \nabla B \|_{L^2}^2  \| \nabla \Psi \|_{L^2} ^2 $$
We also know from the Hodge decomposition that 
$$  \| A \nabla \Psi \|_{L^{2}}^2 =  \| \nabla D \|_{L^{2}}^2 + \| \nabla E \|_{L^{2}}^2 + \| \nabla H \|_{L^{2}}^2 $$ 
and since $H$ is harmonic, the classical monotonicity of $\frac{1}{r^2} \int_{\mathbb{D}_r} \left\vert\nabla H\right\vert^2$ leads to
$$ \| \nabla H \|_{L^{2}\left(\mathbb{D}_r\right)}^2 \leq  r^2 \| \nabla H \|_{L^{2}\left(\mathbb{D}\right)}^2 \leq  r^2 \| A \nabla \Psi \|_{L^{2}\left(\mathbb{D}\right)}^2.  $$
By Proposition \ref{propriviereform}, Let $\eta : \mathbb{D} \to \mathcal{M}_m(\R)$  be such that
$$ PA = I_m + \eta \text{ and }  \sum_{i,j}\left\Vert \eta_{i,j} \right\Vert_{L^\infty(\mathbb{D})}^2 \leq K\eps_0 .$$
Notice that $ \sum_{i,j}\left\Vert \eta_{i,j} \right\Vert_{L^\infty(\mathbb{D})}^2 $ is the square of a submultiplicative norm. This implies that
$$ \left( PA \right)^{-1} = I_m + \tilde{\eta} \text{ where }  \tilde{\eta} = \sum_{i=1}^{+\infty} (-1)^i \eta^i \text{ satisfies }\sum_{i,j}\left\Vert \tilde{\eta}_{i,j} \right\Vert_{L^\infty(\mathbb{D})}^2 \leq \frac{K\eps_0}{\left(1-\sqrt{K\eps_0}\right)^2}$$
and we obtain that $ \nabla \Psi = PA\nabla \Psi + \tilde{\eta} PA\nabla \Psi $
\begin{equation*} \begin{split} 
 \| \nabla \Psi \|_{L^{2}\left(\mathbb{D}_{r}\right)}^2 \leq & 4 \|  P  A \nabla \Psi \|_{L^{2}\left(\mathbb{D}_{r}\right)}^2 + 4  \|  \tilde{\eta} P A\nabla  \Psi \|_{L^{2}\left(\mathbb{D}_{r}\right)}^2 \\
\leq &\kappa(K\eps_0)  \| P A \nabla  \Psi \|_{L^{2}\left(\mathbb{D}_{r}\right)}^2 = \kappa(K\eps_0) \| A \nabla  \Psi \|_{L^{2}\left(\mathbb{D}_{r}\right)}^2  \\  
 = & \kappa(K\eps_0)\left(  \| \nabla E \|_{L^{2}\left(\mathbb{D}\right)}^2 + \| \nabla D \|_{L^{2}\left(\mathbb{D}\right)}^2 + \| \nabla H \|_{L^{2}\left(\mathbb{D}_{r}\right)}^2  \right) \\
 \leq & \kappa(K\eps_0)\left( 2CK \eps_0 + r^2 \right)  \| A \nabla \Psi \|_{L^{2}\left(\mathbb{D}\right)}^2 
\end{split}
\end{equation*}
where $ \kappa(K\eps_0) = 4 \left(1 + \frac{K\eps_0}{\left(1 - \sqrt{K \eps_0}\right)^2}\right)$ choosing $\eps_0$ small enough, independent of $m$  such that $ \kappa(K\eps_0)  2CK \eps_0 \leq \frac{1}{8}$ and then $r$ small enough such that $  \kappa(K\eps_0)  r^2 \leq \frac{1}{8} $, we obtain 
$$  \| \nabla \Psi \|_{L^{2}\left(\mathbb{D}_{r}\right)}^2 \leq \frac{1}{2} \| \nabla \Psi \|_{L^{2}\left(\mathbb{D}\right)}^2 . $$
By  iteration we obtain a constant $\alpha>0$ such that
\eqref{peqpropriviere} occurs. To complete the proof, we use a Cauchy-Schwarz inequality for $\sum$ 
$$ r^{-\alpha}\left\vert \Delta \psi_i \right\vert = \left\vert \sum_j \Omega_{i,j} \left( r^{-\alpha} \nabla \psi_j \right) \right\vert \leq \left(\sum_j \Omega_{i,j}^2\right)^{\frac{1}{2}} \left( r^{-2\alpha }\vert \nabla \Psi \vert^2 \right)^{\frac{1}{2}} $$
and a Cauchy-Schwarz inequality for $\int$
$$ \left( \int_{\mathbb{D}_r(p) } r^{-\alpha} \left\vert \Delta \psi_i \right\vert \right)^2 \leq \sum_j \Vert \Omega_{i,j} \Vert_{L^2(\mathbb{D})}^2 \int_{\mathbb{D}_r(p)} r^{-2\alpha }\vert \nabla \Psi \vert^2  $$
so that \eqref{peqpropriviere} implies \eqref{peqpropriviere2}. 

\medskip

\textbf{Step 2:} Let $1 < p_1 < p_{\infty} < +\infty$. There is a constant $\eps = \eps(p_1,p_\infty)>0$ and $C = C(p_1,p_\infty)>0$ and $\alpha = \alpha(p_1,p_\infty)$ such that for any $p_1 \leq p \leq p_\infty$,
\begin{equation}\label{peqproprivierep} \sup_{x\in \mathbb{D}_{\frac{1}{2}} ; 0<r<\frac{1}{4}} \sum_{i}  r^{-\frac{2\alpha}{p}} \Vert \nabla \Psi_{i}  \Vert_{L^{2p}\left(\mathbb{D}_r(x)\right)}^2 \leq C \sum_{i}   \Vert \nabla \Psi_{i}  \Vert_{L^{2p}\left(\mathbb{D}\right)}^2  \end{equation}
so that letting $1<q<2$ such that $p = \frac{q}{2-q}$
\begin{equation}\label{peqpropriviere2q}  \sum_{i=1}^m  \sup_{x\in \mathbb{D}_{\frac{1}{2}} ; 0<r<\frac{1}{4}} \left( \int_{\mathbb{D}_r(x) } r^{-\frac{\alpha}{p}} \left\vert \Delta \psi_i \right\vert^q \right)^{\frac{2}{q}} \leq C \| \Omega \|_{L^{2}\left(\mathbb{D}\right)}^2 \| \nabla \Psi \|_{L^{2p}\left(\mathbb{D}\right)}^2 \end{equation}

\medskip

\textbf{Proof of Step 2:} This proof is similar to the one for Step 1 but much more simple by the use of standard Calder\'on-Zygmund elliptic estimates. We write it to carefully see again the independence with respect to $m$ of the constants $\eps$, $C$ and $\alpha$. We let $\Psi = \Psi^0 + h$ where $h$ is harmonic on $\mathbb{D}_{\rho}(x) \subset \mathbb{D}$ and $\Psi^0 = 0$ on $\partial \mathbb{D}_{\rho}(x)$. In particular, $\Delta \Psi^0 = \Omega \nabla \Psi$. Along the proof, the constant $C$ can increase but still depending only on $p_0$ and $p_\infty$. Then
\begin{equation*} \begin{split} \Vert \nabla \Psi_i^0 \Vert_{L^{2p}(\mathbb{D}_{\rho}(x))}^2 \leq & C \Vert \Delta \Psi_i^0 \Vert_{L^q(\mathbb{D}_{\rho}(x))}^2  \leq C\left( \int_{\mathbb{D}_{\rho}(x)} \sqrt{\sum_j \vert \Omega_{i,j} \vert^2}^q \sqrt{\sum_j \vert \nabla \Psi_j \vert^2}^q \right)^{\frac{2}{q}} \\
\leq & C \int_{\mathbb{D}}\vert \Omega_{i,j} \vert^2 \left( \int_{\mathbb{D}_{\rho}(x)} \left(\sum_j\vert \nabla \Psi_j \vert^2\right)^{\frac{q}{2-q}} \right)^{\frac{2-q}{q}}  \end{split} \end{equation*}
so that
\begin{equation} \label{eq:elliptic+holder}
\sum_i \Vert \nabla \Psi_i^0 \Vert_{L^{2p}(\mathbb{D}_{\rho}(x))}^2  \leq C \Vert \vert \Omega \vert \Vert_{L^2(\mathbb{D})}^2 \left\Vert \sum_j\vert \nabla \Psi_j \vert^2 \right\Vert_{L^p(\mathbb{D}_{\rho}(x))} \leq C \eps \sum_j \Vert \nabla \Psi_j \Vert_{L^{2p}(\mathbb{D}_{\rho}(x))}^2 
\end{equation}
Standard estimates on harmonic functions yield
$$ \Vert  \vert \nabla h_i \vert^2 \Vert_{L^p(\mathbb{D}_r(x))} \leq C \left(\frac{r}{\rho}\right)^{\frac{2}{p}} \Vert  \vert \nabla h_i \vert^2 \Vert_{L^p(\mathbb{D}_\rho(x))} $$
so that
\begin{equation*} \begin{split} \sum_i \Vert \nabla \Psi_i \Vert_{L^{2p}(\mathbb{D}_r(x))}^2 \leq & 2C \eps \sum_j \Vert \nabla \Psi_j \Vert_{L^{2p}(\mathbb{D}_{\rho}(x))}^2 + 2C \left(\frac{r}{\rho}\right)^{\frac{2}{p}} \sum_j \Vert  \vert \nabla h_j \vert^2 \Vert_{L^p(\mathbb{D}_\rho(x))}  \\
\leq & \left( 2C\eps \left(1 + 2C \left(\frac{r}{\rho}\right)^{\frac{2}{p}} \right) + 2C \left(\frac{r}{\rho}\right)^{\frac{2}{p}} \right) \sum_j \Vert \nabla \Psi_j \Vert_{L^{2p}(\mathbb{D}_{\rho}(x))}^2 \end{split} \end{equation*}
Then, the standard iteration argument gives \eqref{peqproprivierep}. Now, we prove \eqref{peqpropriviere2q}. From a Cauchy-Schwarz inequality, we obtain
$$  \vert \Delta \Psi^0_i \vert^q \leq \left(\sum_j \vert \Omega_{i,j} \vert^2 \right)^{\frac{q}{2}} \vert \nabla \Psi \vert^q $$
so that for $\mathbb{D}_r(x)\subset \mathbb{D}_{\frac{3}{4}}$, a H\"older inequality yields
$$\left( \int_{\mathbb{D}_r(x)} \vert \Delta \Psi^0_i \vert^q\right)^{\frac{2}{q}} \leq \left(\int_{\mathbb{D}_{\frac{3}{4}}}\sum_j  \vert \Omega_{i,j} \vert^2\right) \left(\int_{\mathbb{D}_r(x)} \vert \nabla \Psi \vert^{\frac{2q}{2-q}}\right)^{\frac{2-q}{q}} $$
Letting $\frac{q}{2-q} = p$ we easily obtain
$$  \left(\int_{\mathbb{D}_r(x)} \vert \nabla \Psi \vert^{\frac{2q}{2-q}}\right)^{\frac{2-q}{q}} = \Vert \vert \nabla \Psi \vert^2 \Vert_{L^{p}} \leq \sum_i \Vert \vert \nabla \psi_i \vert^2 \Vert_{L^{p}} = \sum_i \Vert \vert \nabla \psi_i \vert \Vert_{L^{2p}}^2 $$
so that from Step 2, 
$$ \sup_{\mathbb{D}_r(x) \subset \mathbb{D}_{\frac{3}{4}}} \left( r^{-\frac{\theta}{2 p}} \Vert \vert \Delta \Psi_i^0 \vert \Vert_{L^{q}(\mathbb{D}_r(x))}  \right)^2 \leq K \left(\int_{\mathbb{D}_{\frac{3}{4}}}\sum_j  \vert \Omega_{i,j} \vert^2\right)\sum_{i=1}^m \Vert \vert\nabla \Psi_i\vert \Vert^2_{L^{2p}(\mathbb{D}_{\frac{3}{4}})} $$
and we deduce \eqref{peqpropriviere2q}.

\medskip

\textbf{Step 3:} By Morrey elliptic estimates on \eqref{peqpropriviere2}, \eqref{eq:Lpestimate} holds for some $p_1 > 1$. Notice that in general \eqref{eq:Lpestimate} implies \eqref{eq:Lqestimate} by a triangle inequality. Indeed
$$ \Vert \vert \nabla \Psi \vert^2 \Vert_{L^{\frac{p}{2}}\left(\mathbb{D}_{\frac{1}{2}}\right)} \leq \sum_{i=1}^m \Vert \vert \nabla \psi_i \vert^2  \Vert_{L^{\frac{p}{2}}\left(\mathbb{D}_{\frac{1}{2}}\right)} \leq C_p \| \Omega \|_{L^{2}\left(\mathbb{D}\right)}^2 \| \nabla \Psi \|_{L^{2}\left(\mathbb{D}\right)}^2$$
Using again the equation, $\Delta \Psi = \Omega \nabla \Psi$, and a Cauchy-Schwarz inequality for $1 < q < 2$,
$$ \Vert  \Delta \psi_i \Vert_{L^{q}} = \left\Vert \sqrt{\sum_j \vert \Omega_{i,j} \vert^2} \vert \nabla \Psi \vert \right\Vert_{L^q} \leq \left\Vert  \sqrt{ \sum_j \vert \Omega_{i,j}  \vert^2 } \right\Vert_{L^2} \Vert \vert \nabla \Psi \vert \Vert_{L^{\frac{2q}{2-q}}}  $$
so that setting $q$ such that $p = \frac{q}{2-q}$,
$$ \sum_i \Vert  \Delta \psi_i \Vert_{L^{q}}^2 \leq C_0 \Vert \Omega \Vert_{L^2}^2  \Vert \vert \nabla \Psi \vert^2 \Vert_{L^{\frac{q}{2-q}}}.$$
Now complete the proof of Proposition \ref{propriviereuhlenbeck} by bootstrapping the uniform (with respect to $p_1 \leq p \leq p_\infty$) Morrey elliptic estimate \eqref{peqpropriviere2q}. We let $p \geq p_1$ and $p_\infty = p+1$. 
We set
$$ \frac{1}{2p_2} = \frac{1}{q_1} - \frac{1}{2-\frac{\alpha q_1}{2p_1}} = \frac{1}{2p_1} - \frac{\alpha}{4(1+p_1)-2\alpha}. $$
From \eqref{peqpropriviere2q}, Sobolev embedding in Morrey spaces \cite[Theorem 3.2]{ADAMS} yields for some choice of $\nu_1 >0$:
$$ \sum_{i=1}^m \sup_{\mathbb{D}_r(x) \subset \mathbb{D}_{\frac{3}{4} - \nu_1}} r^{-\frac{\theta}{(1+ p_1)p_2}} \Vert \vert \nabla \psi_i \vert \Vert_{L^{2p_2}(\mathbb{D}_r(x))}^2   \leq K_1 \sum_{i=1}^m \Vert \vert \nabla \psi_i\vert \Vert^2_{L^{2p_1}(\mathbb{D}_{\frac{3}{4}})}. $$
We now iterate all the previous argument with
$$ \frac{1}{2p_{k+1}} =  \frac{1}{2p_k} - \frac{\theta}{4(1+p_k)-2\theta}. $$
$$ \nu_k = 2^{-(k+1)} $$
to obtain for any $k$ such that $p_k \leq p_\infty$ 
$$ \sum_{i=1}^m  \Vert \vert \nabla \psi_i \vert \Vert_{L^{2p_{k}}(\mathbb{D}_{\frac{3}{4}- \sum_{j=1}^{k-1} \nu_j})}^2   \leq K_1 \cdots K_{k-1} \sum_{i=1}^m \Vert \vert\nabla \psi_i\vert \Vert^2_{L^{2p_1}(\mathbb{D}_{\frac{3}{4}})} $$
and the proof of the Proposition is complete.
\end{proof}

\subsection{Proof of Theorem \ref{theo:Linftyestimateofgradpsi}}

In the previous subsection, we cannot expect to have better estimates with constants independent of the dimension by bootstrapping the system of equations. Moreover, even if we do it for a fixed integer $m$, we cannot expect to have $L^\infty$ estimates on $\left\vert \nabla \Psi \right\vert^2$ if we only consider the structure equation of proposition \ref{propriviereuhlenbeck} as first noticed in \cite{sharptopping}

Therefore, we have to use that the systems of equations we consider are associated to harmonic maps. We write the equation of eigenmaps
$$ \Delta \Psi = \frac{\vert \nabla \Psi \vert^2_\Lambda}{\vert\Lambda \Psi \vert^2} \Psi $$
in complex coordinates: denoting $f_z = \frac{f_x-if_y}{2}$ and $f_{\bar z} = \frac{f_x+if_y}{2}$, we have that
$$ -\partial_{\bar z}\partial_z \psi_i = \sum_j  \frac{\lambda_j \partial_{\bar z} \psi_j \partial_z \psi_j}{\vert \Lambda \Psi \vert} \nu_i . $$
Using that $\nu \cdot \partial_z \Psi  = 0$, we obtain
$$ -\partial_{\bar z}\partial_z \Psi = \left(\partial_{\bar z} \nu \cdot \partial_z\Psi \right)  \nu  $$
and using the matrix notation, we also obtain
$$ -\partial_{\bar z}\partial_z \Psi = \left(\partial_{\bar z} \nu \cdot \nu^T - \nu \cdot \partial_{\bar z} \nu^T \right) \cdot \partial_z\Psi . $$
and taking $P$ and $\xi$ associated to the Uhlenbeck gauge:
$$ \nabla^{\perp} \xi = P^T \nabla P + P^T \Omega P  $$
written in complex coordinates
$$ i\partial_{\bar z} \xi = P^T \cdot \partial_{\bar{z}} P + P^T\left( \partial_{\bar z} \nu \cdot \nu^T - \nu \cdot \partial_{\bar z} \nu^T \right) P $$
we obtain
$$ -\partial_{\bar z}\left(P^T\partial_z \Psi\right) = \left(  P^T \cdot \partial_{\bar{z}} P + P^T\left( \partial_{\bar z} \nu \cdot \nu^T - \nu \cdot \partial_{\bar z} \nu^T \right) P \right) P^T \partial_z \Psi  $$
so that
\begin{equation}
\partial_{\bar z} \alpha = \omega \alpha \text{ where } \alpha:= P^T \partial_z \Psi \text{ and } \omega:= i \partial_{\bar z} \xi.
\end{equation}

As soon as $\omega$ is somewhat small in $L^{2,1}$, we obtain a priori estimates independent of the dimension of the target manifold thanks to the following claim. Let's first define the adapted norm of matrices in this context:
for $A \in \mathcal{M}_m(L^{\infty}(\mathbb{C},\mathbb{C}))$ we set the norm
$$ \Vert A \Vert = \left\vert \left(\Vert A_{i,j} \Vert_{L^\infty}\right)_{1\leq i,j\leq m}  \right\vert_{2,2} $$
where $\vert \cdot \vert_{2,2}$ is the matrix norm associated to the $l^2$ norm on vectors $\vert \cdot \vert_2$
$$ \vert B \vert_{2,2} = \max_{X \in \R^m\setminus \{0\}} \frac{\vert AX\vert_2}{\vert X \vert_2} = \sqrt{\rho\left( B^T B \right)} $$
and $\rho$ is the spectral radius of the symmetric matrix $B^T B$. In particular, the constant map $I_m$ satisfies $\Vert I_m \Vert = 1$.

\begin{cl} \label{cl:createsolution}
There is $ \eps_1 >0$ such that for any $m \geq 1$ and any $\omega : \mathbb{C} \to \mathcal{M}_m(\mathbb{C})$ such that
$$ \forall 1\leq i,j \leq m, \sum_{i,j} \| \omega_{i,j} \|_{L^{2,1}\left(\mathbb{C}\right)}^2 \leq \eps_1^2  $$
then there is a solution $A \in \mathcal{M}_m(L^{\infty}(\mathbb{C},\mathbb{C}))$ of 
$$ \partial{\bar{z}} A = \omega A$$
such that $A^{-1} = \sum_{k=0}^{+\infty} (I_m -A)^k \in \mathcal{M}_m(L^{\infty}(\mathbb{C},\mathbb{C}))$ and $\max( \Vert A \Vert,\Vert A^{-1} \Vert ) \leq 2$
\end{cl}

\begin{proof}
We set a map $T : \mathcal{M}_m(L^{\infty}(\mathbb{C},\mathbb{C})) \to \mathcal{M}_m(L^{\infty}(\mathbb{C},\mathbb{C}))$ defined as 
$$ T(A)_{i,j}(z) = \sum_{k=1}^m \int_{\mathbb{C}} \frac{1}{\pi(z-w)} \omega_{i,k}(w) A_{k,j}(w)dw$$
for $A\in\mathcal{M}_m(L^{\infty}(\mathbb{C},\mathbb{C}))$, $1\leq i,j\leq m$ and $z \in \mathbb{C}$. Let $\eps >0$. Let $1\leq i,j \leq m$. We let $z_{i,j} \in \mathbb{C}$ be such that
$$ \Vert T(A)_{i,j} \Vert_{L^\infty} \leq \vert T(A)_{i,j}(z_{i,j})\vert + \eps \Vert A_{i,j} \Vert_{L^\infty} $$
where for $X \in \mathbb{R}^m$ and $1 \leq i \leq m$
\begin{equation*}
\begin{split} \left\vert \sum_{j=1}^m T(A)_{i,j}(z_{i,j}) X_j \right\vert^2 & = \left\vert \sum_{k=1}^m \int_{\mathbb{C}} \frac{1}{\pi(z_{i,j}-w)} \omega_{i,k}(w) \sum_{j=1}^m A_{k,j}(w) X_j dw \right\vert^2 \\
& \leq \left\Vert \frac{1}{\pi(z_{i,j} - \cdot)} \right\Vert_{L^{2,\infty}}^2 \sum_{k=1}^m \left\Vert \omega_{i,k} \right\Vert_{L^{2,1}}^2 \sum_{k=1}^m \left(\sum_{j=1}^m \Vert A_{k,j} \Vert_{L^\infty} \vert X_j \vert \right)^2 \\
& \leq \left\Vert \frac{1}{\pi z} \right\Vert_{L^{2,\infty}}^2  \sum_{k=1}^m \left\Vert \omega_{i,k} \right\Vert_{L^{2,1}}^2  \Vert A \Vert^2 \vert X \vert_2^2
\end{split} \end{equation*}
where the first inequality holds by duality $L^{2,\infty} / L^{2,1}$ and a Cauchy-Schwarz inequality, and the last inequality is given by definition of $\Vert A \Vert$ and the invariance by translation of the $L^{2,\infty}$ norm. Summing over $i$ gives
$$ \left\vert \left( T(A)_{i,j}(z_{i,j}) \right)_{1 \leq i,j \leq m} \right\vert_{2,2} \leq \left\Vert \frac{1}{\pi z} \right\Vert_{L^{2,\infty}} \sqrt{ \sum_{1\leq i,j \leq m} \left\Vert \omega_{i,k} \right\Vert_{L^{2,1}}^2 } \Vert A \Vert $$
so that
$$ \Vert T(A)_{i,j} \Vert_{L^\infty} \leq \left( \left\Vert \frac{1}{\pi z} \right\Vert_{L^{2,\infty}} \sqrt{ \sum_{1\leq i,j \leq m} \left\Vert \omega_{i,k} \right\Vert_{L^{2,1}}^2 } + \eps\right) \Vert A \Vert.$$
We now choose $\eps = \frac{1}{8}$ and $\eps_1$ such that $\left\Vert \frac{1}{\pi z} \right\Vert_{L^{2,\infty}} \eps_1 \leq \frac{1}{8} $. so that the map 
$$ \tilde{T}(A) = T(A) + I_m $$ 
is a $\frac{1}{4}$-contraction map on $\left(\mathcal{M}_m(L^{\infty}(\mathbb{C},\mathbb{C})), \Vert \cdot \Vert\right)$. We let $A \in \mathcal{M}_m(L^{\infty}(\mathbb{C},\mathbb{C})) $ the unique fixed point of $\tilde{T}$. Then by definition of $T$ as a convolution, we have that
$$ \partial_{\bar{z}} A = \omega A.$$
Moreover, $A = I_m + T(A)$ implies that
$$ \Vert A-I_m \Vert = \Vert T(A) \Vert \leq \frac{1}{4} \Vert A \Vert \leq \frac{1}{4} \Vert A-I_m \Vert + \frac{1}{4}\Vert I_m \Vert $$
so that with $\Vert I_m \Vert = 1$, 
$$ \Vert A-I_m \Vert \leq \frac{1}{3} $$
In particular, for almost every $z \in \mathbb{C}$, 
$$ \left\vert A(z)-I_m \right\vert_{2,2} \leq \frac{1}{3} $$
so that $ A^{-1}(z) = \sum_{k=0}^{+\infty} \left(I_m - A(z)\right)^k $ exists and $\Vert I_m - A^{-1} \Vert \leq \frac{1}{2}$.
\end{proof}

\begin{cl} \label{clalphazbaromegaalpha}
There is $ \eps_1 >0$ and a constant $C_1$ such that for any $m \geq 1$ and any $\omega \in L^{2}\left( \mathbb{D},\mathcal{M}_m(\mathbb{C})\right)$ such that
$$ \forall 1\leq i,j \leq m, \sum_{i,j} \| \omega_{i,j} \|_{L^{2,1}\left(\mathbb{D}\right)}^2 \leq \eps_1^2  $$
then for any $\alpha \in L^2(\mathbb{D},\mathbb{C}^m)$. satisfying $ \partial_{\bar z} \alpha = \omega \alpha $  we have for all $z\in \mathbb{D}$
$$ \left\vert \alpha(z) \right\vert^2
\leq \sum_{i=1}^m  \| \alpha_i  \|_{L^{\infty}\left(\mathbb{D}_{\frac{1-\vert z \vert}{2}}(z)\right)}^2   
 \leq \frac{C_1}{\left(1-\vert z \vert\right)^2}  \int_{\mathbb{D}} \left\vert \alpha \right\vert^2  $$
\end{cl}

\begin{proof}
We set $\beta = A^{-1} \alpha$ where $A$ is given by Claim \ref{cl:createsolution} for the map $\omega : \mathbb{C} \to \mathcal{M}_m(\mathbb{C})$ extended by $0$ outside $\mathbb{D}$. Then
$$ \partial_{\bar{z}} \beta = \left( \partial_{\bar{z}} A^{-1} \right) \alpha + A^{-1} \partial_{\bar{z}} \alpha = - A^{-1} \left(\omega A \right)A^{-1} \alpha + A^{-1} \omega \alpha = 0 $$
and we obtain that every coordinate of $\beta$ is a holomorphic function on $\mathbb{D}$. In particular, we can write for $z \in \mathbb{D}$
$$ \vert \beta_j(z) \vert^2 \leq \Vert \beta_j \Vert_{L^\infty\left(\mathbb{D}_{\frac{1-\vert z \vert}{2}}(z)\right)}^2 \leq  \frac{C}{\left(1-\vert z \vert\right)^2} \Vert \beta_j \Vert_{L^2(\mathbb{D})}^2 $$
Then for $w \in \mathbb{D}_{\frac{1-\vert z \vert}{2}}(z)$ and $1\leq i \leq m$,
$$ \vert \alpha_i(w) \vert^2 \leq\left\vert \sum_{j=1}^m A_{i,j}(z) \beta_j(z) \right\vert^2 \leq \left(\sum_{j=1}^m \Vert A_{i,j} \Vert_{L^\infty} \Vert \beta_j \Vert_{L^\infty\left(\mathbb{D}_{\frac{1-\vert z \vert}{2}}(z)\right)}\right)^2 $$
so that
\begin{equation*}\begin{split} \sum_{i=1}^m \Vert \alpha_i \Vert_{L^\infty\left(\mathbb{D}_{\frac{1-\vert z \vert}{2}}(z)\right)}^2 \leq \sum_{i=1}^m \left(\sum_{j=1}^m \Vert A_{i,j} \Vert_{L^\infty} \Vert \beta_j \Vert_{L^\infty}\right)^2 \leq \Vert A \Vert^2  \sum_{j=1}^m \Vert \beta_j \Vert_{L^\infty\left(\mathbb{D}_{\frac{1-\vert z \vert}{2}}(z)\right)}^2 \\
\leq \Vert A \Vert^2 \frac{C}{\left(1-\vert z \vert\right)^2} \sum_{j=1}^m \Vert \beta_j \Vert_{L^2(\mathbb{D})}^2 \leq \frac{\Vert A \Vert^2 \Vert A^{-1} \Vert^2C}{\left(1-\vert z \vert\right)^2} \sum_{j=1}^m \Vert \alpha_j \Vert_{L^2(\mathbb{D})}^2  \end{split} \end{equation*}
and the claim follows for $C_1 = 16C$ because $\max\left( \Vert A \Vert, \Vert A^{-1} \Vert\right) \leq 2$.
\end{proof}

\begin{prop} \label{propLinftyestimateofgradpsi}
There is $0 <\eps_2 \leq \eps_1$ and a constant $C_2$ such that for any $\Psi : \mathbb{D} \to \mathcal{E}_{\Lambda}$ harmonic map into the ellipsoid
$$ \mathcal{E}_{\Lambda} = \{ X \in \R^{m} ; \left\vert X \right\vert_{\Lambda} = 1 \} $$
where $\Lambda = (\lambda_1,\cdots,\lambda_{m})$ and $\lambda_1\leq \cdots \leq \lambda_m$, such that setting 
$$ \beta = \frac{\left\vert \nabla \Psi \right\vert_{\Lambda}^2}{\left\vert \Lambda \Psi \right\vert^2} \text{ and } \nu_i = \frac{\lambda_i \psi_i}{\left\vert  \Lambda \Psi \right\vert},$$ 
we assume that
$$ \int_\mathbb{D} \left\vert \nabla \Psi \right\vert^2 \leq 1 \text{ and } \int_\mathbb{D} \beta \leq 1 $$
and that
$$ \int_\mathbb{D} \left\vert \nabla \nu \right\vert^2 \leq \eps_2$$ then
\begin{equation} \label{eqconsequenceclaimalpha} \| \left\vert \nabla \Psi \right\vert^2 \|_{L^{\infty}\left(\mathbb{D}_{\frac{1}{4}}\right)} 
  \leq C_2  \int_{\mathbb{D}} \left\vert \nabla \Psi \right\vert^2  \end{equation}
and
\begin{equation} \label{eqconsequenceclaimalpha2} \Vert \vert \Delta \Psi \vert \Vert_{L^2\left(\mathbb{D}_{\frac{1}{4}}\right)}^2 = \sum_i \| \left\vert \Delta \Psi_i \right\vert \|_{L^{2}\left(\mathbb{D}_{\frac{1}{4}}\right)}^2 \leq 2 C_2  \int_{\mathbb{D}} \left\vert \nabla \Psi \right\vert^2  \int_\mathbb{D} \left\vert \nabla \nu \right\vert^2  \end{equation}
\end{prop}

\begin{proof}
Thanks to Claim, \ref{clalphazbaromegaalpha}, setting $\alpha = P^T \partial_{z} \Psi$, we easily obtain \eqref{eqconsequenceclaimalpha} as soon as we check that $\omega_{i,j} = i \partial_{\bar z} \xi_{i,j}$ satisfies the $L^{2,1}$ assumptions. It suffices to prove that
$$ \forall 1\leq i,j \leq m, \| \nabla \xi_{i,j} \|_{L^{2,1}\left(\mathbb{D}_{\frac{1}{2}}\right)}^2 \leq \eps_1^2 $$
This comes from the equation we deduce by taking the curl of 
$$ \nabla^\perp \xi =   \left(P^T \nu\right) \nabla\left(P^T \nu\right)^T -  \nabla\left(P^T \nu\right) \left(P^T \nu\right)^T - P^T \nabla P $$
so that we have
$$ \Delta \xi =  2 \nabla^\perp \left(P^T \nu\right) . \nabla\left(P^T \nu\right)^T - \nabla^\perp P^T . \nabla P $$
Setting $\xi = H + \tilde\xi$, where $\tilde\xi = 0$ on $\partial \mathbb{D}$ and $H$ is a harmonic function,  we obtain by a Wente inequality (see e.g \cite{helein2}, Theorem 3.4.1) that
$$ \| \nabla \tilde{\xi}_{i,j} \|_{L^{2,1}\left(\mathbb{D}\right)} \leq C_0 \left(   \| \nabla \left(P^T \nu \right)_i \|_{L^{2}\left(\mathbb{D}\right)}  \| \nabla \left(P^T \nu \right)_j \|_{L^{2}\left(\mathbb{D}\right)} + \sum_k  \| \nabla P_{k,i} \|_{L^{2}\left(\mathbb{D}\right)} \| \nabla P_{k,j} \|_{L^{2}\left(\mathbb{D}\right)}  \right)   $$
and we have that
\begin{equation*} 
\begin{split}
\left\vert \nabla \left(P^T \nu \right)_i \right\vert^2 = & \left\vert \nabla \left(\sum_k P_{k,i} \nu_k \right) \right\vert^2 \\
 \leq & \sum_k \left\vert \nabla P_{k,i} \right\vert^2 \nu_i^2 + \sum_k P_{k,i}^2  \left\vert \nabla\nu_i\right\vert^2 + 2 \sum_{k,l} \left\vert\nabla P_{k,i} \right\vert \left\vert P_{l,i} \right\vert \left\vert\nabla \nu_l \right\vert \left\vert \nu_{k} \right\vert \\
 \leq & 2 \left( \sum_k \left\vert \nabla P_{k,i} \right\vert^2 \nu_i^2 + \left\vert \nabla \nu_i \right\vert^2  \right)
\end{split}
\end{equation*}
so that
\begin{equation*} 
\begin{split} \| \nabla \tilde{\xi}_{i,j} \|_{L^{2,1}\left(\mathbb{D}\right)}^2 \leq & 8 C_0^2  \int_{\mathbb{D}}\left( \sum_k \left\vert \nabla P_{k,i} \right\vert^2 \nu_i^2 + \left\vert \nabla \nu_i \right\vert^2  \right) \int_{\mathbb{D}}\left( \sum_k \left\vert \nabla P_{k,j} \right\vert^2 \nu_j^2 + \left\vert \nabla \nu_j \right\vert^2  \right)  \\
&+ 2 C_0^2\left(\sum_k  \| \nabla P_{k,i} \|_{L^{2}\left(\mathbb{D}\right)}^2 \right)\left( \sum_k \| \nabla P_{k,j} \|_{L^{2}\left(\mathbb{D}\right)}^2  \right)
\end{split}
\end{equation*}
and 
$$ \sum_{i,j} \| \nabla \tilde{\xi}_{i,j} \|_{L^{2,1}\left(\mathbb{D}\right)}^2 \leq  4 C_0^2 \left( \int_{\mathbb{D}} \sup_{i}  \sum_k \left\vert \nabla P_{k,i} \right\vert^2  + \int_{\mathbb{D}} \left\vert\nabla \nu \right\vert^2 \right)^2 + C_0^2 \| \nabla P \|_{L^{2}\left(\mathbb{D}\right)}^4  $$
Since $H_{i,j}$ is harmonic, we have a universal constant $\tilde{C}_0$ such that
$$  \| \nabla H_{i,j} \|_{L^{2,1}\left(\mathbb{D}_{\frac{1}{2}}\right)}^2 \leq  \tilde{C}_0^2 \| \nabla H_{i,j} \|_{L^{2}\left(\mathbb{D}\right)}^2 \leq \tilde{C}_0^2 \| \nabla \xi_{i,j} \|_{L^{2}\left(\mathbb{D}\right)}^2  $$
so that
\begin{equation*} 
\begin{split} \sum_{i,j} \| \nabla \xi_{i,j} \|_{L^{2,1}\left(\mathbb{D}_{\frac{1}{2}}\right)}^2& \leq  2  \sum_{i,j} \| \nabla \tilde{\xi}_{i,j} \|_{L^{2,1}\left(\mathbb{D}_{\frac{1}{2}}\right)}^2 + 2 \sum_{i,j} \| \nabla H_{i,j} \|_{L^{2,1}\left(\mathbb{D}_{\frac{1}{2}}\right)}^2 \\
& \leq   4 C_0^2 \left( \int_{\mathbb{D}}\sup_{i}  \sum_k \left\vert \nabla P_{k,i} \right\vert^2  + \int_{\mathbb{D}} \left\vert\nabla \nu \right\vert^2 \right)^2 + C_0^2 \| \nabla P \|_{L^{2}\left(\mathbb{D}\right)}^4 + \tilde{C}_0^2  \| \nabla \xi \|_{L^{2}\left(\mathbb{D}\right)}^2 \\
& \leq \tilde{C}_1 \eps_2
 \end{split}
\end{equation*}
for some constant $\tilde{C}_1$. We then obtain \eqref{eqconsequenceclaimalpha}. Then we have that
\begin{equation*}
\begin{split} \sum_i \| \left\vert \Delta \Psi_i \right\vert \|_{L^{2}\left(\mathbb{D}_{\frac{1}{4}}\right)}^2  = \sum_{i} \int_{\mathbb{D}_{\frac{1}{4}}} \left\vert \sum_k \Omega_{i,k} \nabla \Psi_k \right\vert^2 \\
 \leq \sum_{i,k} \int_{\mathbb{D}_{\frac{1}{4}}} \left\vert \Omega_{i,k} \right\vert^2 \left\vert \nabla \Psi \right\vert^2 \leq \| \left\vert \Omega \right\vert \|_{L^{2}\left(\mathbb{D}_{\frac{1}{4}}\right)}^2 \| \left\vert  \nabla \Psi \right\vert^2 \|_{L^{\infty}\left(\mathbb{D}_{\frac{1}{4}}\right)} 
\end{split}
\end{equation*}
so that  \eqref{eqconsequenceclaimalpha} implies  \eqref{eqconsequenceclaimalpha2}.
\end{proof}

\subsection{Remark : a weighted Rivière equation specific to harmonic eigenmaps}
We give a new equation on $\Psi = \Lambda^{\frac{1}{2}}\Phi : \mathbb{D}\to \mathbb{S}^{m-1}$ transforming the harmonic eigenmap equation
$$ \Delta \Phi = \beta \Lambda \Phi \text{ and } \beta = \frac{\vert \nabla \Phi \vert^2_\Lambda}{\vert \Lambda \Phi \vert^2} $$
into 
$$ \Delta \Psi = \beta \Lambda \Psi = \frac{\vert \nabla \Psi \vert^2}{\vert \Psi \vert^2_\Lambda} \Lambda \Psi = \frac{ \Lambda \Psi \nabla \Psi^T - \nabla \Psi \left(\Lambda \Psi\right)^T}{\vert \Psi \vert^2_\Lambda} \nabla \Psi + \frac{ \nabla \Psi \nabla \vert \Psi \vert_\Lambda}{\vert \Psi \vert_\Lambda} $$
so that
$$ - div\left( f \nabla \Psi \right) =  \widetilde{\Omega} \cdot \left(f \nabla \Psi\right)$$
where $f = \vert \Psi \vert_{\Lambda}$ satisties $f \geq \lambda_1$ and
$$  \widetilde{\Omega} =  \frac{ \Lambda \Psi \nabla \Psi^T - \nabla \Psi \left(\Lambda \Psi\right)^T}{f^2} $$
satisfies $ \widetilde{\Omega}_{i,j} =  \widetilde{\Omega}_{j,i}$ for $1\leq i,j \leq m$.

\begin{rem}The complex version of the equation is
\begin{equation*}\begin{split} - \partial_{\bar{z}}\left( \vert \Psi \vert_\Lambda \partial_z \Psi \right) = - \partial_{\bar{z}} \vert \Psi \vert_\Lambda \partial_{z} \Psi + \vert \Psi \vert_\Lambda \frac{1}{4} \Delta \Psi =   - \partial_{\bar{z}} \vert \Psi \vert_\Lambda \partial_{z} \Psi + \vert \Psi \vert_\Lambda^{-1} \left( \Lambda\Psi \cdot \partial_{z} \Psi^T \right) \partial_{\bar{z}} \Psi \\ 
 - \partial_{\bar{z}} \vert \Psi \vert_\Lambda \partial_{z} \Psi + \vert \Psi \vert_\Lambda^{-1} \left( \Lambda\Psi \cdot \partial_{z} \Psi^T - \partial_z \Psi \cdot (\Lambda \Psi)^T \right) \partial_{\bar{z}} \Psi + \vert \Psi \vert^{-1}_\Lambda \frac{1}{2}\partial_{\bar z} \vert \Psi \vert^2_\Lambda \partial_z \Psi  = \omega \vert \Psi \vert_\Lambda \partial_{\bar z} \Psi \end{split} \end{equation*}
where we complexified $ \widetilde{\Omega}$
$$\tilde{\omega} =  \frac{ \Lambda \Psi \cdot \partial_{z} \Psi^T - \partial_{z} \Psi \cdot \left(\Lambda \Psi\right)^T}{f^2} $$
so that for $\alpha = f \partial_z \Psi$
$$ - \partial_{\bar{z}} \alpha = \tilde{\omega} \cdot \bar{ \alpha }. $$
Because of the conjugate, this equation is very different from the equation $\partial_{\bar{z}} \alpha = \tilde{\omega} \cdot  \alpha $ and we cannot apply of Claim \ref{cl:createsolution} or Claim \ref{clalphazbaromegaalpha} to obtain
 to obtain $L^\infty$ estimates on $\alpha$.
\end{rem}
By straightforward computations, we have that
\begin{equation}\label{eq:smallnessassump} \frac{1}{2}\int_{\mathbb{D}} f \vert  \widetilde{\Omega} \vert^2 + \int_{\mathbb{D}} \frac{\vert \nabla f \vert^2}{f} = \int_{\mathbb{D}} \vert \nabla \Psi \vert ^2 \frac{\vert \Lambda \Psi \vert^2_\Lambda}{\vert \Psi \vert^3_\Lambda} = \int_{\mathbb{D}} \beta \frac{\vert \Lambda \Psi \vert^2_\Lambda}{\vert \Psi \vert_\Lambda}.\end{equation}
We apply Lemma II.1 in \cite{dalioriviere} which is a weighted version of Proposition \ref{propgauge} to obtain

\begin{prop} \label{propgauge2}
$P\in W^{1,2}(\mathbb{D},SO(m))$ such that $P - I_m \in W^{1,2}_0(\mathbb{D},SO(m)) $ and $\xi \in W^{1,2}(\mathbb{D},so(m))$ such that
$$ -div\left( f P^T \nabla \Psi \right) = \nabla^\perp \xi \cdot  P^T \nabla \Psi  $$
and 
$$ \int_\mathbb{D} f \vert \nabla P \vert^2 + \int_\mathbb{D} \frac{\vert \nabla \xi \vert^2}{f} \leq 3 \int_{\mathbb{D}} f \vert \widetilde{\Omega} \vert^2 $$
\end{prop}

\begin{rem}
In Lemma II.1 in \cite{dalioriviere}, contrary to Proposition \ref{propgauge}, we minimize $\int_{\mathbb{D}} f \vert P\widetilde{\Omega} - \nabla P \vert^2$ in $\{ P \in W_f^{1,2}(\mathbb{D},SO(m)), P=I_m \text{ on } \partial\mathbb{D} \}$. Then, the Dirichlet boundary condition holds on $P$, not on $\xi$.
\end{rem}

\begin{prop}[\cite{dalioriviere}, part III] \label{propriviereformweight}
There is a constant $K$ and a constant $\eps_0>0$ such that for any $m \in \mathbb{N}^\star$, a solution $\Psi: \mathbb{D} \to \mathbb{R}^{m}$ of 
$$ -div\left( f \nabla \Psi\right) = \widetilde{\Omega} . f \nabla \Psi  $$ 
where $\sqrt{f}\widetilde{\Omega} \in L^2(\mathbb{D},so(m))$ such that
$$ \int_{\mathbb{D}} f \vert \widetilde{\Omega} \vert^2 \leq \eps_0 $$
can be written
$$ -div(f A \nabla \Psi) = \nabla^\perp B \nabla \Psi $$
where $A - I_m \in W_0^{1,2}\cap L^{\infty}\left( \mathbb{D}, \mathcal{M}_m(\R) \right)$, $B \in W^{1,2}( \mathbb{D}, \mathcal{M}_m(\R))$ and
$$ \sum_{i,j} \Vert (PA - I_m)_{i,j} \Vert_{L^\infty(\mathbb{D})}^2  + \int_{\mathbb{D}} f \left\vert \nabla A \right\vert^2 
+  \int_{\mathbb{D}} \frac{\left\vert \nabla B \right\vert^2}{f} \leq K \int_{\mathbb{D}} f \left\vert \Omega \right\vert^2$$
where $P$ is given by Proposition \ref{propgauge2}
\end{prop}

\begin{rem} We refer to \cite{dalioriviere}, part III and the proof of Proposition \ref{propriviereform} for details concerning independence with respect to the number of coordinates. In this proposition, one can see that the required smallness assumption is stronger than the one we need in Theorem \ref{theo:Linftyestimateofgradpsi}.
That's why we do not complete the analysis here.
\end{rem}

\end{document}